\documentclass{amsart}
\usepackage{amsmath,amssymb,amsthm,amscd,amsopn,amsfonts,amsxtra} 
\usepackage{latexsym,array,marvosym,url}
\usepackage[mathscr]{eucal}
\usepackage{graphicx,psfrag,epsfig,color}
\newtheorem{theorem}{Theorem}[section] 
\newtheorem{lemma}[theorem]{Lemma}

\newtheorem{assumption}[theorem]{Assumption}

\theoremstyle{remark}
\newtheorem{remark}[theorem]{Remark}

\theoremstyle{definition}

\begin{document}

\date{Received in final form February 26, 2011}

\title%[Spectral multipliers for Schr\"odinger operators]
{Spectral multipliers for  Schr\"odinger operators} 

\author{Shijun Zheng}

\address{Department of Mathematical Sciences\\
Georgia Southern University\\
Statesboro, GA 30460-8093}
\address{ and}
\address{Department of Mathematics \\
          University of South Carolina  \\
         Columbia, SC 29208\\ 
         \quad}

\email{szheng@GeorgiaSouthern.edu}

\urladdr%[Shijun Zheng]
{http://math.georgiasouthern.edu/\symbol{126}{szheng}}
%\thanks{The author was supported in part by DARPA grant HM1582-05-2-0001 } 

%\keywords{spectral multiplier, Schr\"odinger operator} 

\begin{abstract}
We prove a sharp H\"ormander multiplier theorem for 
\mbox{Schr\"odinger}
operators $H=-\Delta+V$ on $\mathbb{R}^n$.   
 The result is obtained under certain condition on a 
weighted $L^\infty$ estimate, coupled with a weighted $L^2$ {estimate} for $H$, which is a weaker condition than that
 for nonnegative operators via the heat kernel approach. 
Our approach is elaborated in one dimension with  potential $V$ belonging to certain critical 
weighted $L^1$ class. Namely, we assume that \mbox{$\int (1+|x|) |V(x)|dx$} is finite 
and $H$ has no resonance at zero. In the resonance case we assume 
$\int (1+|x|^2) |V(x)| dx$ is finite.
\end{abstract}

\maketitle 

\definecolor{orange}{rgb}{0.995, 0.75, 0.35}
\definecolor{purple}{rgb}{0.7, 0.2, 0.5}
\definecolor{royalblue}{rgb}{0.2, 0.7, 0.8}
\definecolor{darkgreen}{rgb}{0.2,0.725,0.25}
\def\al{\alpha}
\def\de{\delta}
\def\eps{\epsilon}
\def\ga{\gamma}
\def\lam{\lambda}
\def\om{\omega}
\def\veps{\varepsilon}
\def\vphi{\varphi}
\def\De{\Delta}
\def\Ga{\Gamma}
\def\Lam{\Lambda}
\def\Om{\Omega}

\def\sh{\sinh}
\def\ch{\cosh}
\def\th{\tanh}
\def\sech{\mathrm{sech}}
\def\iy{\infty}
\def\inv{^{-1}}
\def\pa{\partial}
\def\supp{\mathrm{supp}\,}
\def\sgn{\textrm{sgn}}

\newcommand{\cal}{\mathcal}
\newcommand{\ud}{\mathrm{d}}
\newcommand{\bvt}{\big\vert}
\newcommand{\mf}{\mathfrak}
\newcommand{\one}{\mathbf{1}}
\newcommand{\la}{\langle}
\newcommand{\ra}{\rangle}
\newcommand{\nd}{\noindent}
\newcommand{\td}{\tilde}
\newcommand{\vs}{\vspace}
\newcommand{\n}{\newline}
\newcommand{\np}{\newpage}
\newcommand{\hB}{\hfill$\Box$}
\newcommand{\hr}{\hookrightarrow}
\newcommand{\da}{\dagger}
\newcommand{\crr}{{\color{red}$\dagger$}}

\newcommand{\rk}{{\bf Remark.}\ \ }
\newcommand{\defi}{{\bf Definition:}\ \ }
\newcommand{\nota}{{\bf Notation:}\ \ }

\newcommand{\Z}{\mathbb{Z}}
\newcommand{\R}{\mathbb{R}}
\newcommand{\C}{\mathbb{C}}
\newcommand{\N}{\mathbb{N}}
\newcommand{\bS}{\mathbb{S}}
\newcommand{\bfR}{\mathbf{R}} %\Leftscissors, \Cutleft 

\section{Introduction}\label{S1}
Let  $H=-\De+V $  be a Schr\"{o}dinger operator on $\R^n$, 
 where $\De=\sum^n_{j=1}\frac{\pa^2}{\pa x^2_j}$ and $V$ is real-valued.  
  In this paper we are concerned with proving a spectral multiplier theorem on $L^p$ spaces for $H$ and
we then consider potentials in some critical class $L^1_1$ in one dimension, 
where $V$ may {\em not} be positive. As is well known, spectral multiplier theorem plays a significant role in 
harmonic analysis and PDEs \cite{A94,C88,He90a,DOS02,D01,CuS01,
Sch05b,OZ08, CGT82, CS88}. 

For a Borel measurable function $\phi$: $\R\to\C$ 
we define $\phi(H)=\int \phi(\lam) dE_\lam$ by functional calculus, where
$H=\int \lam dE_\lam$ is the spectral resolution of the selfadjoint operator $H$ acting in $L^2(\R^n)$. 
The spectral multiplier problem is %for a bounded measurable $\mu$,
to find sufficient condition on a bounded function $\mu$ on $\R$ (with minimal smoothness) 
so that $\mu(H)$ is bounded on $L^p(\R^n)$, $1<p<\iy$.

In the Fourier case, i.e., $V=0$, H\"ormander \cite{Hor60} essentially proved 
(for radial multipliers) the multiplier theorem 
 on $L^p(\R^n)$, under the condition that the scaling-invariant local Sobolev norm on $\mu$ is finite
for $s>n/2$, 
\begin{align}\label{e:mu-Xs*}
\Vert\mu\Vert_{W^s_{2,sloc}}:= \sup_{t>0}\Vert\mu(t\cdot)\chi\Vert_{W_2^s(\R)}<\iy \,.\end{align} 
Here $\chi\in C^\iy_0(\R\setminus\{0\})$ is a fixed $C^\iy$-smooth function with compact support away from zero 
and $W_2^s$ denotes the usual Sobolev space endowed with the norm 
$\Vert f\Vert_{W^s_2}=\Vert (1-\De)^{s/2}f\Vert_2$. 
The proof in \cite{Hor60}  mainly 
requires that %regularity
the kernel $K_\mu(x,y)$ of $\mu(-\De)$ satisfy
\begin{equation}\label{e:hor-cond}
\int_{|x-\bar{y}|>2|y-\bar{y}|} |K_\mu(x,y)-K_\mu(x,\bar{y})|dx\le C
\end{equation}
for all $y$, $\bar{y}$ (for the weak $(1,1)$ estimate). 
However, the %integral 
regularity condition in (\ref{e:hor-cond}) is invalid for $H$ when $V\neq 0$.

 For $V\ge 0$,  Hebisch \cite{He90a} proved a multiplier theorem with $s>\frac{n+1}{2}$ based on heat kernel estimates.
His  approach 
was essentially to control the low energy part of $\mu(H)$ 
by a pointwise decay of the kernel, see (\ref{e:phi-dec}). 
This  heat kernel approach has been recently developed in proving sharp 
multiplier theorems (with $s>n/2$) in various settings for positive elliptic %higher-order differential 
operators on manifolds or metric spaces % measures 
\cite{A94,CM96,DM99}, %where a weighted $L^2$ estimate 
see \cite{DOS02} for a comprehensive survey and the references therein. 

The question remains open for general $V$ 
where the heat kernel estimates may not hold. % if $V$ has nonzero negative part.  
In this paper %\cite{Z06a} develop an approach to treat multiplier problem by generalizing the 
%H\"ormander-Hebisch method
we formulate a H\"ormander type spectral multiplier theorem (Theorem \ref{t:m(H)Lp})
for general $H$ on $\R^n$. We show that Theorem \ref{t:m(H)Lp}
%the sharp H\"ormander spectral multiplier theorem 
is true if the two weighted estimates in Assumption \ref{a:wei-phi-dec}, 
namely a weighted $L^2$ estimate (in high energy) and an integral  form of pointwise decay estimate
(in low energy), are satisfied for $H$.  
In Sections \ref{s:wei-infty-hi}--\ref{s:wei-L2-lowhi}, we elaborate the approach in one dimension by considering potentials in the %critical 
class %in one dimension 
 $L^1_\ga:=\{ f: \int (1+|x|)^\ga |f(x)| dx<\infty\}$, $\ga=1, 2$.  

For a (continuous) function $\phi$, 
let $\phi(H)(x,y)$ denote the kernel of $\phi(H)$, $x,y\in\R^n$ and let $\lam_j=2^{-j/2}$, $j\in\Z$. 
 By $\phi\in {X}(\Om)$, %\dot{W}^s_2(\Om)$  
where $\Om\subset \R$ and $X$ is a function space on $\R$,  we mean that $\phi\in {X}$ %\dot{W}^s_2(\R)$ 
 and has support in $\Om$. Throughout this paper $c$ or $C$ will denote an absolute constant 
%$\lesssim$ the notation $\le c$,
 and $\chi_\Om$ the characteristic function on the set $\Om$.

\begin{assumption}\label{a:wei-phi-dec} Assume that $H$ satisfies the following two estimates. 

\textup{(a)} (Weighted $L^2$ estimate)
There exists some $s>n/2$ so that for all $j$ and $\phi\in {W}^{s}_2([\frac{1}{4},1]\cup[-1,-\frac{1}{4}])$,
 \begin{equation}\label{e:wei-L2-ineq}  \sup_y \Vert | x-y|^s \phi(\lam_j^2 H)(x,y)\Vert_{L^2_x} 
\le c\lam_j^{s-n/2}\Vert \phi\Vert_{{W}_2^{s}}  \,.%\Vert  \phi(\xi^2)\Vert_{\dot{X}_*^{n/2+\eps}},  %$H}^{s}}  
%\quad \forall j
\end{equation}  
%where $c_s=c(\Vert \phi\Vert_{W_2^{s}([\frac14,1])})$. %$X^s=W_2^s([1/4,1])$, %or $C^s(\R)$  0,\dots,[n/2]+1$

\textup{(b)} (Weighted $L^\iy$ estimate) %(\ref{e:poly-dec-zeta}) 
There exist a finite measure $d\zeta$ and $0<\eps\le 1$ so that for all $x,y$, $j$ and $\phi\in W^{n+\eps}_2([-1,1])$,
\begin{equation}\label{e:phi-dec-zeta} % \phi(t H)(x,y) \le
 |\phi(\lam_j^2 H)(x,y)|\le   c%(\Vert \phi\Vert_{W_2^{n+\eps}})
  \lam_j^{-n} \int_{\R^n} (1+ \lam_j^{-1} |x-y-u|)^{-n-\eps}d\zeta(u)\,,
 \end{equation}where $c=c(\Vert\phi\Vert_{W_2^{n+\eps}})$.% only 
\end{assumption} 

The assumption is intrinsic in the sense that it only depends on $H$ and does not depend on the multiplier
$\mu$. Note that when $V\ge 0$, Hebisch \cite{He90a} essentially used in the proof the following pointwise decay 
% \quad $V\ge 0$
\begin{equation}\label{e:phi-dec} % \phi(t H)(x,y) \le
 |\phi(\lam_j^2 H)(x,y)|\le   c \lam_j^{-n} (1+ \lam_j^{-1} |x-y|)^{-n-\eps} , 
\end{equation}
%to control the low energy part of $H$
which is implied by the upper Gaussian bound for $e^{-tH}(x,y)$. 
Assumption \ref{a:wei-phi-dec} (b) is a much 
weaker condition than (\ref{e:phi-dec}). %[He90a] %for \mbox{general} $V$%(negative or general V)
% special case when $\zeta=\de$
When $V$ is negative, the decay in (\ref{e:phi-dec}) %for $\Phi_j(H)(x,y)$ ($j>j_0$) 
does not hold, not even for $V$ being a Schwartz function, cf. \cite{OZ06,Z08}.

\begin{theorem}\label{t:m(H)Lp} Suppose $H$ satisfies Assumption \ref{a:wei-phi-dec} for some $s>n/2$.  
If $\Vert \mu\Vert_{W^s_{2,sloc}}<\iy$, 
then $\mu(H)$ is bounded on $L^p(\R^n)$, $1<p<\infty$ 
 and has weak type $(1,1)$.
\mbox{Moreover,} 
\begin{equation}\label{e:mu-w11-scal}
\Vert \mu(H)\Vert_{L^1\to weak\text{-}L^1} \le c\Vert\mu\Vert_{W^s_{2,sloc}} \,. 
\end{equation} 
\end{theorem}

That the critical exponent $\frac{n}{2}$ is sharp is well-known in the literature \cite{%Hor60,
C91,Se89,DOS02}. Note that the condition in (\ref{e:mu-Xs*}) implies $\mu\in L^\iy$ %\cap C$ 
%$\Vert\mu\chi\Vert_\iy\le \Vert\mu\chi\Vert_{H^s}$ 
by Sobolev embedding %$H^s(\R)\hr L^\iy(\R)$ whenever $s>1/2$.
\begin{equation}\label{e:mu-infty-Ws}
\Vert \mu\Vert_\iy\le c\Vert \mu\Vert_{W^s_{2,sloc}}
\end{equation}
whenever $s>1/2$. Also, note that %it gives rise to 
one has an equivalent norm
for  $\Vert\cdot\Vert_{W^s_{2,sloc}}$ %$\Vert\mu\Vert_{X^s_*}^\chi\approx $ 
if in (\ref{e:mu-Xs*}) $\chi$ is replaced with any other $\vphi$ in $C^\iy_0(\R\setminus\{0\})$.

%\vs{.1152in}
\begin{remark} From the proof given in Section \ref{s:pf-w11} we easily observe that
Theorem \ref{t:m(H)Lp} %and Theorem \ref{t:m(H)-B-F} 
actually holds for any self-adjoint operator $L$ in place of $H$ 
that satisfies Assumption \ref{a:wei-phi-dec} (a)  and\\ 
%and with Assumption \ref{a:wei-phi-dec} (b) replaced with: 
(b') There exist $d\zeta_k\in M$, $k\in\Z$, $M$ the set of finite measures, 
 with $0<\eps\le 1$ and
$\sum_k\Vert \zeta_k\Vert_M<\iy $, so that for all $x,y,j$ %and $\phi\in C^\iy([-1,1])$
\begin{equation}\label{e:phi-dec-rhoj-zetaj} 
 |\Phi(\lam_j^2 L)(x,y)|\le   
 c\sum_{k,\pm} \lambda_k^{-n}(1+ {\lambda_k}^{-1} |\cdot|)^{-n-\eps} *d\zeta_k(\pm x\pm y),
 \end{equation}
where $\Phi\in C^\iy([-1,1])$ is given as in (\ref{e:Phi-phij}),
 $f*d\zeta(x)=\int f(x-u)d\zeta(u)$ is the usual convolution. 
\end{remark}
 
Applying Theorem \ref{t:m(H)Lp} to the one dimensional $H_V:=-d^2/dx^2+V$, we obtain the following
theorem.
\begin{theorem}\label{t:m(H)-V-Lp} Suppose $V$ is in $L^1_1(\R)$ and assume that there is no resonance at zero.  If for some $s>1/2$,
$\Vert \mu\Vert_{W^s_{2,sloc}}$ is finite, %:=\sup_{t>0}\Vert \mu(t \cdot)\chi\Vert_{W^s_2}<\iy ,
then the conclusions of Theorem  \ref{t:m(H)Lp} hold. 
Furthermore, the conclusions also hold true for all $V\in L^1_2(\R)$. 
\end{theorem}

A typical example for $\mu$ is  $\mu_\ga(\xi)=|\xi|^{i\ga}$, $\ga\in \R$. Hence 
$H_V^{i\ga}$ is bounded on $L^p$, $1<p<\iy$ and maps $L^1$ to weak-$L^1$. 

Let $\{\Phi, \varphi_j\}\in C_0^\infty ({\R}) $ be a dyadic system satisfying 
%\mbox{$\forall x \in \mathbf{R}$,}
$\;  \supp \;\Phi \subset \{ x: |x|\le 1\}$, $\supp\; \varphi
\subset \{ x: \frac14\le |x|\le 1\} $ and 
\begin{align}
&\sum_{j=-\infty}^\infty \varphi_j(x) =1,    \qquad \forall x\neq 0,     \label{eq:dyadic-id}\\
&\Phi(x)  +\sum_{j=1}^\infty \varphi_j(x) =1,    \quad   \forall x , \label{e:Phi-phij}
\end{align}
where $\varphi_j(x) =\varphi( 2^{-j} x)$, and note that $\Phi(x)\equiv 1$ on $[-\frac12,\frac12]$.
%i.e., $\supp(1-\Phi)=\{|x|\ge \frac12\}$

Using the dyadic system above we will make the high and low energy cutoffs of $\mu(H)$ in the proof of 
Theorem \ref{t:m(H)Lp}.
As in \cite{JN94,Z06b,DP05}, %in Section  \ref{s:pf-w11} 
we can also define  $B^{\al,q}_p(H)$ and $F^{\al,q}_p(H)$,
the Besov spaces and Triebel-Lizorkin spaces associated with $H$.  
We can show that the sharp spectral multiplier theorem also hold on these spaces, see the statement in 
\mbox{Theorem \ref{t:m(H)-B-F}.}  

\subsection{Weighted estimates for the kernel of $\phi(2^{-j}H_V)$} Let $V\in L^1_1(\R)$ and assume $0$ is not a resonance or let $V\in L^1_2(\R)$ in general. From \cite{DT79} or Section \ref{s:wei-infty-hi},  $H_V$ has resonance at 0 means that the Wronskian 
vanishes at $0$, i.e., $\nu:=W(0)=0$.
From \mbox{Theorem \ref{t:m(H)Lp}} and the remark that follows we know that 
the main technical difficulty in proving Theorem \ref{t:m(H)-V-Lp} is to verify the two weighted estimates in %(\ref{e:wei-L2-ineq}), 
Assumption \ref{a:wei-phi-dec} (a), (b').

 The proofs of (\ref{e:wei-L2-ineq}) and (\ref{e:phi-dec-rhoj-zetaj})  for $H_V$  require some new and refined formulas and asymptotic estimates
for $m_\pm(x,k)$, the modified Jost functions,  and $t(k)$, $r_\pm(k)$, the associated transmission and reflection coefficients.
The main tools are Volterra integral equations for $m_\pm(x,k)$ as well as its Fourier transforms. 
These are motivated by and developed from the treatment in \cite{DT79}.

For the $L^\iy$ estimates in (\ref{e:phi-dec-rhoj-zetaj}) for  
the low energy, we use  Wiener's lemma %(in the resonance and no-resonance)
 in order to prove the existence of finite measures $d\zeta_k$, which are
actually $L^1$ functions up to a delta measure, see \cite{GS04} for a similar treatment when
considering the dispersive estimates for $H_V$.

For the $L^2$ estimates in (\ref{e:wei-L2-ineq}) for the high energy, we prove (\ref{e:wei-L2-ineq}) for $1/2<s<1$ by interpolating
between the cases $s=0$ and $s=1$, which can be viewed as Plancherel formula for $D^s\phi_j$ with respect to 
the Fourier transform associated to $H_V$. 

The remaining of the paper is organized as follows. In Section 2 we prove the weak (1,1) estimate
for general $H$ under the hypothesis in Assumption \ref{a:wei-phi-dec}.
Sections \ref{s:wei-infty-hi} to \ref{s:wei-L2-lowhi} are devoted to the proof of Theorem \ref{t:m(H)-V-Lp}, which is
quite long verification of the estimates in (\ref{e:wei-L2-ineq}) and (\ref{e:phi-dec-rhoj-zetaj})
%Assumption \ref{a:wei-phi-dec} (a), (b')
in one dimension. %for $V\in L^1_1(\R)$.  
In certain cases it involves 
delicate and subtle technicalities. 

\section{Proof of weak-(1,1) boundedness}\label{s:pf-w11} 

In this section we mainly give the proof of Theorem \ref{t:m(H)Lp}. Since $\mu\in L^\iy$, $\mu(H)$ is bounded on $L^2$.
Hence by interpolation and duality it is sufficient to show that $\mu(H)$ has weak type (1,1), which will follow from Lemma \ref{l:j-horm-cond},
 Lemma \ref{l:Phi-iint-dec} and Calder\'on-Zygmund decomposition. 
The proof is a modification of the arguments in \cite{He90a} and \cite{DM99}.
%cf. also \cite{Fe70,C88} %for some original idea %for considering strongly singular integrals
%in different contexts.  %cf also C88 
%The higher energy is controlled by the $L^2$ weighted inequality. 
Let $\{\Phi,\vphi_j\}$ be as in (\ref{eq:dyadic-id}), (\ref{e:Phi-phij}). 
Write $\mu_j=\mu\vphi_j$,  $\Phi_j(x)=\Phi(2^{-j}x)$, $j\in\Z$. 
%The technical ingredients we need are the following two lemmas
\begin{lemma} \label{l:j-horm-cond} Let $H$ satisfy Assumption \ref{a:wei-phi-dec} (a) with $s>n/2$. 
  Let $y\in I$, $I\subset\R^n$ a cube with length $t=\ell(I)=2^{-j_I/2}$, $j_I\in\Z$.
Then\\
(a) For all $j\ge j_I$, 
\[\int_{|x-y|\ge 2t} | \big(\mu_j(1-\Phi_{j_I})\big)(H)(x,y)| dx \le c(2^{j/2}t)^{\frac{n}{2}-s} 
\Vert \mu\Vert_{W^s_{2,sloc}}\,.\]%X_*^s} $$  
(b)\[
\int_{|x-y|\ge 2t} \sum_{j=-\infty}^\infty| \big(\mu_j(1-\Phi_{j_I})\big)(H) (x,y)| dx \le c\Vert \mu\Vert_{W^s_{2,sloc}}\,.
\]
%In particular, 
%\[ \int_{|x-y|\ge 2t} \sup_{j\in\Z} | \mu_j(H)(1-\Phi_{j_I})(H)(x,y)| dx \le C\]
\end{lemma}

\begin{proof} Inequality (a) is consequence of Assumption \ref{a:wei-phi-dec} (a) and Schwarz inequality. %the weighted inequalities. 
Let $\td{\mu}_j={\mu}_j(1-\Phi_{j_I})$.  
  We have for $s>n/2$, $j\ge j_I$, 
\begin{align*}
&\int_{|x-y|\ge 2t}  |\td{\mu}_j(H)(x,y)| dx \\
=&\int_{|x-y|\ge 2t}|x-y|^{-s} |x-y|^s  |\td{\mu}_j(H)(x,y)| dx \\
\le&c_{n,s}(2^{j/2}t)^{n/2-s} \Vert \mu\Vert_{W^s_{2,sloc}}\,. %[\frac14,1]}%X_*^s}   %\qquad{\color{red}checked}
\end{align*}
 
(b) is an easy consequence of (a). Note that since $\supp\vphi_j\subset \{2^{j-2}\le|\xi|\le2^j\}$ %\cup [-2^{j},-2^{j-2}]$ 
and  $\supp(1-\Phi_{j_I})\subset \{|\xi|\ge 2^{j_I-1}\}$, %,\iy)\cup (-\iy,-2^{j_I-1}]$ %counting only positive half line
it follows that $\td{\mu}_j(\xi)=0$ if $j\le j_I-1$.
\end{proof}

\begin{lemma}\label{l:Phi-iint-dec} 
Let $H$ satisfy Assumption \ref{a:wei-phi-dec} (b) with some finite measure $d\zeta$
and $\eps\in (0,1]$. 
Let $y\in I$, $I$ a cube with length $t=\ell(I)= 2^{-j_I/2}$, $j_I\in\Z$ %\textup{(}$| I |=
and volume $|I|$. Then for all $x$ and all $y\in I$
\begin{align*}  |\Phi_{j_I}(H)(x,y)|\le   c |I|\inv \int_{u\in\R^n} \int_{z\in I} 2^{j_In/2} (1+2^{j_I/2}|x-z-u|)^{-n-\eps} dz d\zeta(u)\,.%\int_{\R^n} (1+ \lam^{-1} |x-y-u|)^{-n-\eps}d\zeta(u)
 \end{align*}
\end{lemma}

\begin{proof}  Since $\Phi\in C^\iy([-1,1])\subset W_2^{n+\eps}([-1,1])$, %directly use the fact that 
according to  (\ref{e:phi-dec-zeta}), $\Phi_{j_I}(H)(x,y)$ is dominated by %($\lam_j=2^{-j_I/2}$) 
\begin{align*}
%&\int_u \lam^{-1}\Psi^\vee(\lam^{-1}(x-y-u) ) d\zeta(u)
 &c\int_{\R^n} \lam_j^{-n}(1+\lam_j\inv|x-y-u|)^{-n-\eps}d\zeta(u)\\
\le&c |I|^{-1} \int_{\R^n} \int_{z\in I} 2^{j_In/2} (1+2^{j_I/2}|x-z-u| )^{-n-\eps} dz d\zeta(u), \quad\forall x\in\R^n,  y\in I,
\end{align*}
where $\lam_j=2^{-j_I/2}$ and we observed that for all $x$ and $t=\ell(I)$
\begin{align*}%\label{e:max-min-psi-t}
\sup_{y\in I} (1+ |x-y|/t)^{-n-\eps} %\le \sup_{y\in I} (1+ |x-y|/t)^{-n-\eps}
\le c \min_{y\in I} (1+ |x-y|/t)^{-n-\eps}\le
\frac{c}{|I|}\int_I  (1+ |x-z|/t)^{-n-\eps} dz .
\end{align*}
%\footnote{The max-min ineq basically says that if y runs over any cube J of length $\le 1$ then
%$\max_{y\in J}(1+ |x-y|)^{-n-\eps}\le c\min_{y\in J}(1+ |x-y|)^{-n-\eps}$.
%Since when $y$ runs over $I$ with $\ell(I)=t=\lam=2^{-j_I/2}$, 
%$y/t$ runs over a cube of length $\lesssim 1$, the first ineq in (\ref{e:max-min-psi-t}) holds.}
\end{proof}

\subsection{Proof of the weak-(1,1)}\label{ss:w-11}  
Let $f\in L^1\cap L^2$. For any given $\al>0$, apply the C-Z decomposition to obtain that $f=g+b$
for some $g\in L^1\cap L^2$, %L^\iy$
and $b\in L^1$ with $b=\sum_k b_k$, 
where $\supp\,b_k\subset I_k$, $I_k$ being disjoint %(dyadic) 
cubes in $\R^n$ with lengths $\ell(I_k)$ equal to integer powers of $\sqrt{2}$ and 
\begin{itemize}
\item[(i)] $\displaystyle |g(x)|\le c\al$\qquad a.e. $x$
\item[(ii)] $\displaystyle |I_k|\inv\int_{I_k} |f(x)|dx\le c\al$
\item[(iii)] $\sum_k |I_k|\le c\al\inv\Vert f\Vert_1\,.$
\end{itemize}%\edz{In \cite[p.17]{St93} it is not required $f\ge 0$}

We will prove that  there exists a constant $C$ such that $\forall f\in L^1\cap L^2$,
\begin{align}\label{e:mu-w11}
\vert \{x: |\mu(H)f(x)|>\al\}|\le C\al\inv\Vert f\Vert_1\big(\Vert \mu\Vert_{W^s_{2,sloc}}+
\Vert \mu\Vert_\iy^2+ 1\big).
\end{align}

Since $\mu\in L^\iy$, Chebeshev inequality gives
\begin{align*}
\vert&\{x: |\mu(H)g(x)|>\al/2\}|\le (\al/2)^{-2}\Vert \mu(H)g\Vert_2^2\\ 
\le& c\Vert \mu\Vert_\iy^2\al\inv \Vert f\Vert_1 \,.
\end{align*}

The main task is to deal with the ``bad'' function $b$. 
%\begin{proof} 
Let $\Phi$ %\subset [-1,1]$ 
be as in (\ref{e:Phi-phij}), $\Phi_j(x)=\Phi(2^{-j}x)$.
Write 
\[ \mu(H)b(x)=\sum_k \mu(H)(1-\Phi_{j_k}(H)) b_{k}(x)
+ \sum_k \mu(H)\Phi_{j_k}(H) b_{k}(x),\] 
where $2^{-j_k}= \ell(I_k)^2$. 
Denote by $I_k^*$ the cube having length $5\sqrt{n}$ times the length of $I_k$ with the same
center as $I_k$. We need to show
\begin{align*}
\vert&\{x\in\R^n\setminus \cup_kI_k^*: 
 |\mu(H) b(x)| >\al/2 \}\vert\\
\le&  |\{x\in\R^n\setminus \cup_kI_k^*:  \sum_k |\mu(H)(1-\Phi_{j_k}(H)) b_{k}(x)|>\al/4\}\\
+&    |\{x\in\R^n\setminus \cup_kI_k^*:  |\sum_k \mu(H)\Phi_{j_k}(H) b_{k}(x)|>\al/4\} \\
\le &c\Vert \mu\Vert_{W_{2,sloc}^s}\al^{-1} \Vert f\Vert_1\,.
\end{align*}
%where $b=\sum_k b_k$. %(convergence in $L^1\cap L^q$ so $T b(x)=  \sum_k Tb_{k}$ in $L^q$)
 %This requires dividing the estimates of $\mu(H)$ into high and low energy cut-offs

\nd{\bf a. High energy cut-off.} If $x\notin \cup_kI_k^*$, then $I_k\subset \{y: |y-x|>2\sqrt{n}\,t_k \}$, 
$t_k%=\ell(I_k)
=2^{-j_k/2}$. %being the length of $I_k$. %(checked)
We have %($\supp\,b_k\subset I_k\subset \{y: |y-x|>2\sqrt{n}t_k \}$)
%\footnote{if $y_0\in I_k=\supp\,b_k$ and $x\notin I_k^*$. Let $\bar{y}=$center of $I_k$. Then 
%$d(x,y_0)\ge d(x,\bar{y})-d(\bar{y},y_0)\ge 5\sqrt{n}(\frac{1}{2}t_k)-\sqrt{n}(\frac{1}{2}t_k)\ge 2\sqrt{n}t_k$} 
\begin{equation*}
 \mu(H)(1-\Phi_{j_k}(H ))b_k(x)=\int_{|y-x|>2t_k} \big(\mu(1-\Phi_{j_k})\big)(H)(x,y) b_k(y)dy .
\end{equation*}
 Applying Lemma \ref{l:j-horm-cond} (b) for $s>n/2$, we obtain
\begin{align*}  
&|\{x\notin \cup I_k^*:  |\sum_k \big(\mu(1-\Phi_{j_k})\big)(H )b_k(x)|>\al/4 \}| \\
 \le& c (\al/4)^{-1} \int_{\R^n\setminus\cup I_k^*}  |\sum_k \big(\mu(1-\Phi_{j_k})\big)(H)b_k(x)| dx \\
\le& c \al^{-1} \int  \sum_k |b_k(y)| dy\int_{|y-x|>2t_k} |\big(\mu(1-\Phi_{j_k})\big)(H)(x,y)|dx\\
 %\le&  c\sup_\lam\Vert \mu(\lam\xi)\chi\Vert_{X^s} 
 %\Vert \mu\Vert_{X_*^s}\al^{-1} \int |b(y)|dy 
 \le& c\Vert \mu\Vert_{W_{2,sloc}^s}\al^{-1}\Vert  f\Vert_1,
\end{align*}
where we note that
\begin{align*}
&\int_{|x-y|>2t_k} |\big(\mu(1-\Phi_{j_k})\big)(H)(x,y)| dx\\
\le&\int_{|x-y|>2t_k} \sum_j |\big(\mu_j(1-\Phi_{j_k})\big)(H)(x,y)| dx
%\le& \sum_{2^j> t_k^2}\int_{|x-y|>2t_k} |\mu_j(H)(x,y)| dx
\le c\Vert \mu\Vert_{W_{2,sloc}^s}. \end{align*}

%because  if $\Phi(x)+\sum_{j=1}^\infty \phi(2^{-j}x)=1$ then for any  $j_0\in\Z$,
%$\Phi(2^{-j_0}x)=1-\sum_{j=j_0+1}^\infty \phi(2^{-j}x)$
%\footnote{The general principle is in [Stein 93, Chap.I, \S 5, Singular integrals]
%If prove weak (1,1) for $f\in L^1\cap L^q$ then interpolation gives $L^p$, $1<p<q$.
%The proof uses \[ \int | Tf|^pdx=p\int_0^\iy \mu\{x:|Tf(x)|>\lam \}\lam^{p-1} d\lam\]
%together with the density of $L^1\cap L^q$ in $L^p$.
%BTW note that if the restricted weak (1,1) holds, then if $f_n\to f$ in $L^1$,
%$Tf_n$ (is Cauchy in measure and) converges in measure to a (unique) $Tf$ 
%Thus Tf can be defined for $L^1$. However in Stein's interpolation proof for $L^p$, $p>1$ we will not need such an extension on $L^1$ }

\nd{\bf b. Low energy cut-off.} %(Local)  adopt the ``duality" method [He90a,DM99]
Since $\mu(H)$ is bounded on $L^2$, the proof is complete if we can show
\begin{equation}\label{e:sum-L2}
\int |\sum_k   \Phi_{j_k}(H)b_k(x)|^2 dx\le c \al  \Vert f\Vert_1. 
\end{equation}
To show this %let $h\in L^2(\R^n)$ and %$2^{-j_k}=\ell(I_k)^2$
let $\rho_j=2^{jn/2}(1+2^{j/2}|\cdot | )^{-n-\eps}$. %\text{is $\sim$ an approximation to the id})
According to Lemma \ref{l:Phi-iint-dec},   $\forall h\in L^2$,
%The first term is approximation to the identity.  So we only deal with the second
%term which we denote by $K_j^+(x,y)$.  
%If $j<0$, ok. So we only deal with $j>0$. 
 %If $h\in L^2(\R^n)$,  $2^{-j_k}\sim \ell(I_k)^2$, $j_k>0$.
\begin{align*} 
&\vert\la  \sum_k \Phi_{j_k}(H)  b_k  , h \ra\vert
=\vert \sum_k \int h(x)dx \int_{y\in I_k} \Phi_{j_k}(H)(x,y)  b_k(y)dy\vert   \\
%\le & \sum_k \int_x |h(x)|dx\int_{y\in I_k}  \lam^{-1}\int_{u\in\R}  |\Psi^\vee(\lam^{-1} (x-y-u)) | d\mu(u) |b_k(y)|dy   \\
%= & \sum_k \int_{y\in I_k}  \int_{u\in\R}  |\Psi^\vee(\lam^{-1} (x-y)-u) |  e^{-c\lam |u|} du b_k(y)dy  \int_x h(x)dx \\
\le & \sum_k   
|I_k |^{-1} \int%_{y\in I_k}
|b_k(y)| dy\int |h(x)|dx\int_{z\in I_k}\int_u  %\lam^{-1}|\Psi^\vee(\lam^{-1} (x-z-u)) | 
\rho_j(x-z-u) d\zeta(u) dz\\   %\footnote{{\color{red}$\dagger$} it is important to gain a factor $|I_k|^{-1}$)}
%& (\text{apply Lemma ``Harnack"  for} \; \Phi_{j_k}(H)(x,y) )\\
\le& %\Vert\Psi^\vee\Vert_1 
c\al %|I_k|^{-1}\Vert b_k\Vert_1
  \int%_{z\in I_k} 
  \sum_k\chi_{I_k}(z)dz  \int (M_{HL} h) (z+u)    d\zeta(u)  \\ 
\le & c \al  \Vert\sum_k \chi_{I_k}\Vert_2 \Vert M_{HL} h* d\td{\zeta} \Vert_2\qquad (\text{$d\td{\zeta}=d\zeta(-\cdot)$ a finite measure})\\
\le & c \al  (\sum_k |I_k|)^{1/2} \Vert h\Vert_2
%\le & C \al  (\al^{-1} \Vert f\Vert_1)^{1/2} \Vert h\Vert_2
\le c \al^{1/2}\Vert f\Vert_1^{1/2} \Vert h\Vert_2, %\quad \forall h\in L^2,
\end{align*}
%\begin{align*}
%\la&  \sum_k \Phi_{j_k}(H)  b_k(x)  , h \ra\\
%\le&  \sum_k \Vert b_k\Vert_1 |I_k|^{-1}\int_x\int_{y\in I_k}  \rho_{t}(x-y)dy  |h(x)| dx\\ 
%& (\text{it is here we use ``Harnack" for} \; \Phi_{j_k}(H)(x,y) )\\
%\le& C \al \sum_k \int_x\int_y \chi_{I_k}(y)  \rho_{t}(x-y)dy  |h(x)| dx\\
%\le& C \al \int_y \sum_k\chi_{I_k}(y)  M h(y) dy\\
%\le & C \al  \Vert\sum_k \chi_{I_k}\Vert_2 \Vert M h\Vert_2\\
%\le & C \al  (\sum_k |I_k|)^{1/2} \Vert h\Vert_2\\
%\le & C \al  (\al^{-1} \Vert f\Vert_1)^{1/2} \Vert h\Vert_2
%= C  \al^{1/2}\Vert f\Vert_1^{1/2} \Vert h\Vert_2
%\end{align*}
which proves (\ref{e:sum-L2}) by duality. %\footnote{{\color{red}$\dagger$} It is important and imperative to gain a factor $|I_k|^{-1}$.}  
We have used the fact that if  $\rho_t= t^{-n}\rho(x/t)$ is any approximation  to the identity
so that $\rho\in L^1(\R^n)$ is positive and decreasing,  then %the ineq holds. 
\[
 \sup_{t>0} |\rho_t*f(x)| \le M_{HL}f(x),
 \]
where $M_{HL}$ denotes the Hardy-Littlewood maximal function on $\R^n$. 

Therefore (\ref{e:mu-w11}) is established. In view of (\ref{e:mu-infty-Ws}), the weak-(1,1) bound %$c\Vert \mu\Vert_{W^s_{2,sloc}}$ 
in (\ref{e:mu-w11-scal}) follows via the same argument above if, for given $\al>0$, instead of decomposing 
$f$ at height $\al$, one decomposes $f$ at height  %a scaling argument 
\mbox{$\al/ \max(\Vert \mu\Vert_{W_{2,sloc}^s}, \Vert \mu\Vert_\iy)$,} see e.g., \cite{CS01} for details.\hB%\end{proof}

%\footnote{Note that in the proof of  w-(1,1) the condition $\int b_k=0$ is not used as in classical cases. The reason is that  the kernel is rough, not having Lipschitz regularity. So the idea of Hebisch-Zheng method is to alternatively use certain weighted average}

\subsection{Besov and Triebel-Lizorkin spaces}\label{ss:b-tl-H}%\input{besov-tl-H.tex}
For a general selfadjoint operator acting on $L^2(\R^n)$, one can define 
the associated Besov and Triebel-Lizorkin spaces \cite{E95,JN94,OZ08}. %BT06}. % associated with $L$.  
Let $H=-\De+V$ on $\R^n$. %as in \mbox{Section 1.} 
Under the same conditions for $H$ and $\mu$ as in Theorem \ref{t:m(H)Lp}, 
we can show that $\mu(H)$ is bounded on these generalized spaces 
(cf. Theorem \ref{t:m(H)-B-F}, where $W^s_2$ is replaced with an %general 
abstract space). 
Like in the Fourier case \cite{BL76,Tr83}, the spectral multiplier theorems on them %on $B_p^{\alpha,q}(H), F_p^{\alpha,q}(H)$ 
are closely related to some of the main results in Littlewood-Paley 
theory for $H$ (interpolation, embedding and identification) 
 \cite{JN94, E96, OZ06, %Cu00, 
 DP05}.

Let $\alpha \in \mathbb{R}$ and $1\le p,q\le\infty$. The homogeneous {\em Besov space} 
associated with $H$, denoted by $%\dot{B}_p^{\alpha,q}:=
\dot{B}_p^{\alpha,q}(H)$, 
is defined to be the completion of the Schwartz class $\cal{S}(\R^n)$, %L_0^2:= \{f\in L^2: 
%$\Vert f\Vert_{\dot{F}_p^{\alpha,q}} <\infty \}$, 
where the norm $\Vert \cdot\Vert_{\dot{B}_p^{\alpha,q}(H)}$ is given by %initially 
%defined for $f\in L^2$: 
\begin{equation}\label{e:B-norm}
\Vert f\Vert_{\dot{B}_p^{\alpha,q}(H)} %:= \Vert f\Vert_{F_p^{\alpha,q}(H)}^{\phi}
= \bigg(\sum_{j=-\infty}^{\infty} 2^{j\alpha q} \Vert \vphi_j(H)f \Vert_p^q\bigg)^{1/q} .
\end{equation}

Similarly, the homogeneous {\em Triebel-Lizorkin space} $\dot{F}_p^{\alpha,q}(H)$ is defined by the norm 
\begin{equation*}%\label{eq:F-norm}
\Vert f\Vert_{\dot{F}_p^{\alpha,q}(H)} %:= \Vert f\Vert_{F_p^{\alpha,q}(H)}^{\phi}
=\Vert \bigg(\sum_{j=-\infty}^{\infty} 2^{j\alpha q} \vert \vphi_j(H)f \vert^q\bigg)^{1/q}\Vert_p\,.
\end{equation*}

%A vector-valued version of the proof for the $L^p$ gives the multiplier theorem on $F(H)$
For $s\in\R$ let $X^s\subset \cal{S}'(\R)$ be a Banach %algebra, i.e., a %complete vector 
space endowed with a norm $\Vert \cdot\Vert_{X^s}$\,, where $\cal{S}'(\R)$ is the space of tempered distributions
 on $\R$.
 %that is closed under multiplication. 
Further assume that $\{X^s\}_{s\in\R}$ satisfies the following properties.
\begin{itemize}
\item[a)] $\displaystyle C^\iy_0(\R)\subset X^s$, \quad $\forall s$
\item[b)] $\displaystyle X^{1/2+\eps}\subset L^\iy(\R)\cap C(\R)$, \quad  $\forall \eps>0$
\item[c)] $\displaystyle \Vert uv\Vert_{X^s}\le c\Vert u\Vert_{X^s} \Vert v\Vert_{X^s}$, \quad $\forall u,v\in X^s$, 
 $s>n/2$.
%\item[d)] $\displaystyle [X^{s_0},X^{s_1}]_\theta=X^s$, $s=(1-\theta)s_0+\theta s_1$\,
%where $[A,B]_\theta$ denotes the usual complex interpolation between $A$ and $B$
\end{itemize}
Examples of $X^s$ include $W^s_p(\R)$, $p\in (1,\iy)$ and $B^{s,q}_p(\R)$, $p,q\in (1,\iy)$, the classical Sobolev and Besov spaces,
see \cite[\S 6.8]{BL76} or \cite{Tr83}.

%\footnote{If $p=\iy$, $W^{s+\eps}_\iy=?F^{s+\eps,2}_\iy\hr B^{s,1}_\iy\hr L^\iy$ so (ii)=(b) maybe ok.
%how about (iii)? not sure about the algebra part. check [Tr83]}

\begin{theorem}\label{t:m(H)-B-F} %{\bf Theorem \ref{th:m(H)F}.} 
Suppose $H=-\De+V$ verifies Assumption \ref{a:wei-phi-dec} (a), (b) with $X^s, X^{n+\eps}$ replacing $W^s_2$,
 $W^{n+\eps}$ respectively.
Let $\mu$ satisfy for some $s>n/2$
\begin{equation*}
\Vert \mu\Vert_{X^s_*}:=\sup_{t>0}\Vert \mu(t\cdot)\chi\Vert_{X^s}<\iy,
\end{equation*}
where $\chi$ is a fixed function in $C^\iy_0(\R\setminus\{0\})$. % with compact support away from zero
%$|\mu^{(k)}(x)|\le C/|x|^k$, $k=0,1,\dots,[\frac{n}{2}]+1$. 
Then $\mu(H)$ extends
to a bounded  operator on $\dot{B}_p^{\alpha,q}(H) $ for $ 1<p<\infty$,
$1\le q\le \infty$, $\alpha \in \R$ and  
$\dot{F}_p^{\alpha,q}(H) $ for $ 1<p<\infty$, $1<q< \infty$, $\alpha \in \R$.
\end{theorem} 

Note that Theorem \ref{t:m(H)Lp} holds under the same hypothesis in Theorem \ref{t:m(H)-B-F} with the same proof 
given in this section. %\edz{the B(H), F(H) here need to explain the homog and inhomog spaces
%$ B(H), F(H)$ or  $\dot{B}(H), \dot{F}(H)$ ?} 
The statement for $B_p^{\alpha,q}(H)$ follows immediately from (\ref{e:B-norm}). %the definition.
To show the statement for $F_p^{\alpha,q}(H)$,  we need 
to prove, as a key step, that the operator $T_\mu:=\{{\mu}_j(H)\}$ maps $L^1(\ell^q)$ continuously to weak-$L^1(\ell^q)$,
where $\mu_j=\mu\vphi_j$ and $T_\mu$ %\{\mu_j(H)\}$ 
is given by $\{f_j\}\mapsto \{{\mu}_j(H)f_j\}$. This can be achieved by 
a vector-valued version of the proof %for the weak-(1,1) estimate
in Subsection \ref{ss:w-11}.  The details are presented in \cite{OOZ}.

Under additional smoothness condition on $V$, one can identify $F_p^{\alpha,q}(H)=F_p^{2\alpha,q}(\R^n)$, which
allows us to obtain the boundedness of $\mu(H)$ on %ordinary 
$F_p^{\alpha,q}$ and $B_p^{\alpha,q}$ spaces on $\R^n$ 
 according to Theorem \ref{t:m(H)-B-F}, cf. \cite{OZ08,Tr83}.

\begin{remark} We would like to mention that 
the boundedness of $\mu(H)$ on $L^p$, $1<p<\iy$ can also be obtained from wave operator method 
\cite{Y95,W99,DF06}.  
However our results give the endpoint estimate $L^1\to\text{weak-}L^1$ 
and also the boundedness for $F^{\al,q}_p$ spaces (including Sobolev space),
which consequently lead to interpolation and embedding %duality %or identification 
results.
The reason is that wave operator method can transfer the integrability 
but somehow lose 
the pointwise information. \end{remark}

\section{Weighted $L^\iy$ estimates: High energy}\label{s:wei-infty-hi} 
Let $V\in L^1(\R)$. 
Then $H_V$ has the form domain $W^1_2(\R)$,
whose absolute continuous spectrum $\sigma_{ac}(H_V)=[0,\iy)$ and singular continuous spectrum is empty. 
The pure point spectrum $\sigma_{pp}(H_V)$ is finite provided that $\int (1+|x|) |V|dx<\infty$. 
%Sections \ref{s:wei-infty-hi} to \ref{s:wei-L2-lowhi} are devoted to the proof of Theorem \ref{t:m(H)-V-Lp}
Let $H_{pp}$ and $H_{ac}$ denote the projections of $H_V$ onto
the pure point and absolute continuous subspaces of $L^2(\R)$ respectively. 
From \cite[C.3]{Si82} we know that the eigenfunctions have exponential decay $\lesssim e^{-c|x|}$, $c>0$.    
\footnote{$A\lesssim B$ stands for the usual notion $A\le c B$ for some absolute constant $c$.}
 %$\sigma_p$ consists of poles of $T(k)$, which lie in the upper half plane 
It follows that $\Vert \mu(H_{pp})f\Vert\le c\Vert f\Vert_p$, $1\le p\le \iy$. Hence 
in view of the remark following Theorem \ref{t:m(H)Lp}, %\ref{r:selfad-H} 
it suffices to verify %Assumption \ref{a:wei-phi-dec} (a) and (b') 
(\ref{e:wei-L2-ineq}) and (\ref{e:phi-dec-rhoj-zetaj}) for $H_{ac}$ in place of $H$. 
As we will show, %in Sections \ref{s:wei-infty-hi} to \ref{s:wei-L2-lowhi}, 
(\ref{e:phi-dec-rhoj-zetaj}) is a result of
Lemma \ref{l:high-ker-dec-mu} and Lemma \ref{l:low-ker-dec-mu}, and (\ref{e:wei-L2-ineq}) is a result of 
interpolation between Lemma \ref{l:wei-L2-mu-lo} and Lemma \ref{l:j-ker-L2}.

\subsection{Kernel formula} 
Let $R_V(z)=(H_V-z)\inv$ be the resolvent of $H_V$, $z\in \C\setminus [0,\iy)$.
For $\phi\in C(\R)$, $\phi(H)$ has the resolvent expression \cite[XIII.6]{RS} 
\begin{align}\label{e:spec-mea-H}
\phi(H_{ac})f(x)=&
\frac{1}{\pi}\int_0^\infty \phi(\lam)\Im R_V(\lam+i0)f d\lam\notag \\ 
=&\frac{1}{2\pi i} \int_{0}^\infty \phi(\lam)[R_V(\lam+i0)-R_V(\lam-i0)] f d\lam.
%=&\frac{1}{\pi i} \int_{-\infty}^\infty \lam\phi(\lam^2) R_V(\lam^2+i0) f d\lam (only valid for R_V(z)=R_0(z))
\end{align}

Let $W(\lam)$ be the Wronskian of $f_+$, $f_-$, then
 for $\lam\neq 0$ %($W(\lam)\neq 0$)
\begin{align*}
& R_V(\lam^2\pm i0)(x,y)=\begin{cases}
\frac{f_+(x,\pm\lam)f_-(y,\pm\lam)}{W(\pm\lam)}& x>y\\
\frac{f_+(y,\pm\lam)f_-(x,\pm\lam)}{W(\pm\lam)}& x<y,\\
\end{cases}
\end{align*}
where $f_\pm(x,z)$ are the Jost functions that solve
for $\Im z\ge 0$
\begin{equation}\label{e:jost-eq}
-f''_\pm(x,z)+V(x)f_\pm(x,z)=z^2f_\pm(x,z)
\end{equation}
and satisfy the asymptotics
\begin{align*}
f_\pm(x,z)\to
\begin{cases}e^{\pm izx} &x\to\pm \iy,\\
\frac{1}{t(z)}e^{\pm izx}+\frac{r_\mp(z)}{t(z)}e^{\mp izx} &x\to\mp \iy,
\end{cases}
\end{align*}
 $t(z), r_\pm(z)$ being the transmission and reflection coefficients
respectively, see \cite{DT79, GS04}.%DF06}
%\footnote{$t(k)f_+(x,k)$ describes scattering from left to right of a plane wave $e^{ikx}$,
%$t(k)f_-(x,k)$ describes scattering from right to left of a plane wave $e^{-ikx}$.  
%$r_-(k)=r_2$ is the reflection coefficient from left to right, $r_+(k)=r_1$ is the reflection coefficient from right to left \cite{W99}}

Let $m_\pm(x,z)=e^{\mp iz x}f_\pm(x,z )$ be the {\em modified Jost functions}.
We obtain, from %the resolvent 
formula (\ref{e:spec-mea-H}) of the spectral measure of $H_{ac}$, that  %if $x>y$
\begin{equation}\label{e:ker-phi-m}
\phi(H_{ac})(x,y)=\frac{1}{2\pi }\int_{-\iy}^\iy \phi(\lam^2)m_+(x,\lam)m_-(y,\lam)
t(\lam)e^{i\lam(x-y)}\, d\lam \,, 
\end{equation}
where $t(\lam)= -2i\lam/W(\lam)$, see e.g. \cite{GS04,OZ08}.\footnote{Since the kernel formula coincides
with the one using Lippmann-Schwinger scattering eigenfunctions,  (\ref{e:ker-phi-m}) is valid for 
both $x>y$ and $x<y$. }
%As a result the restriction $x>y$ can be dropped. %see also [GH98; Lemma 7.4] \input{zheng-goldberg.tex}

\subsection{Fourier transforms of $m_\pm(x,k)$} The following lemma for $m_\pm$, $t$, $r_\pm$ %for $|k|\gtrsim 1$ 
are basically recorded %in  the proof of Lemma \ref{l:high-ker-dec-mu} and other results 
from \cite{DT79}, %\cite[Theorem 2, p.159]{DT79} 
see also \cite{OZ08}. Let $B_\pm(x,y)$ be the pair of functions
 satisfying the Marchenko equations in (\ref{e:B+marchenko}), (\ref{e:B-marchenko}).
\begin{lemma}\label{l:t-r-m-B} Let $V\in L^1_1$. Then
\begin{align*}
&m_+(x,k)=1+\int_0^\infty B_+(x,y)e^{2iky}dy\\
&m_-(x,k)=1+\int^0_{-\infty} B_-(x,y)e^{-2iky}dy\\  %\quad checked Jun 11 08
&t(k)\inv= 1-\frac{1}{2ik}\int_{-\iy}^\iy V(t)m_\pm(t,k)dt \\
%&=1-\frac{1}{2ik}\int V(t)m_-(t,k)dt \\
&\qquad\;\,= 1-\frac{\nu}{2ik}-\frac{1}{2ik}\int_{-\iy}^\iy V(t) dt\int_0^\infty B_+(t,y)(e^{2iky}-1) dy\\   
 &\qquad\;\, =  1-\frac{\nu_0}{2ik}-\frac{1}{2ik}\int_{-\iy}^\iy V(t) dt\int_0^\infty B_+(t,y) e^{2iky} dy\\ 
                  %=& 1-\frac{\nu_0}{2ik}-  \frac{1}{2ik} \int_0^\infty B(y)e^{2iky}  dy \quad (\text{also from 
                  %\cite[Theorem 2, p.159]{DT79} )}\\
                  %=&  1- \frac{1}{2ik}\int_\R V(t)  dt - \frac{1}{2ik}\int_\R V(t_1) \int_{t_1}^\infty h (t_2-t_1,k)V(t_2)   dt_1dt_2
                    %+ \cdots\\     %$\nu_0=\int V$. %$|B(y)|\le c\chi_{(0,\infty)}(y)(\int_{y/2}^\infty |V(s)|ds + \int^{-y/2}_{-\infty} |V(s)|ds) $
&r_\pm(k)t(k)\inv= \frac{1}{2ik}\int_{-\iy}^\iy e^{\mp 2ikt}V(t)m_\mp(t,k)dt , %\qquad checked\\
%&r_-(k)t(k)\inv=\frac{1}{2ik}\int e^{2ikt}V(t)m_+(t,k)dt . %\qquad checked\; Apr.1, 08\\ 
\end{align*}  where $\nu_0=\int_{-\iy}^\iy V(t)dt$ and  
\begin{equation}\label{e:nu-W}
\nu:=W(0)=\int_{-\iy}^\iy V(t)m_+(t,0)dt=\int_{-\iy}^\iy V(t)dt (1+\int_0^\infty B_+(t,y)dy ),\end{equation}
 see \cite[Remark 9, p.152]{DT79}.  \end{lemma} %\footnote{[DF] wrote the formula $m_\pm(x,k)=1+\int_0^\infty B_\pm(x,y)e^{-2iky}dy$}
%\edz{$W(\lam=0)=%[f_+,  f_-](\lam=0)=[m_+,  m_-](\lam=0)$}

Let $\hat{f}(k)=\int f(x)e^{-ikx}dx$ and $g^\vee(x)=\int g(k)e^{ikx}dk$. The following lemma
gives estimates on the Fourier transforms of $m_\pm$, %Lemma \ref{l:m+m-rho} 
which is an easy consequence of Lemma \ref{l:t-r-m-B} and Lemma \ref{l:B-L1} (c). %and the fact that $V\in L^1_1$  % omit the pf 
%Let $\R_+=(0,\iy)$ and $\R_-=(-\iy,0)$
\begin{lemma}\label{l:m+m-rho} Let $V\in L^1_1$. 
Let $x>0, y<0$. Then there exists a constant $c=c(\Vert V\Vert_{L^1_1})$ independent of $x,y$ such that $\forall u$%(Case a)  %A. $resonance$
\begin{align*} 
&|m_+(x,\pm\cdot)^\vee(u)|%c\chi_{u<0} B_+(x,-u/2)\\
%\le c\int_{x-u/2} |V(t)|dt\\
\le 2\pi\de+ c\chi_{\{\pm u<0\}}%e^{\int_x^\iy (t-x)|V(t)|dt}\int_{-u/2}^\iy |V(t)|dt\in 
\rho^+(\mp u/2)\in \R_+\de+L^1(\R_\mp) \\%\qquad\forall u\\
&|m_-(y,\pm\cdot)^\vee(u)|%c\chi_{u<0} B_+(x,-u/2)\\ %\le c\int_{x-u/2} |V(t)|dt\\
\le 2\pi\de+ c \chi_{\{\pm u<0\}}\rho^-(\pm u/2)%e^{\int^y_{-\iy} (y-t)|V(t)|dt}\int^{u/2}_{-\iy} |V(t)|dt
\in \R_+\de+L^1(\R_\mp) , %\qquad\forall u\\
\end{align*}
where $\de$ is the Dirac measure at zero, 
$\rho^+(u)=\int_{u}^\iy |V(t)|dt$, $\rho^-(u)=\int^{u}_{-\iy} |V(t)|dt$, $\R_+=(0,\iy)$ and $\R_-=(-\iy,0)$.
%\footnote{this is a sum of $\de$ and an $L^1$ function ind of $x>0$, $y<0$ if $V\in L^1_1$}  
\end{lemma} 

The next lemma provides series expansions for $t(k)$, $r_\pm(k)$ in the high energy, 
whose proofs will be postponed till the end of this section.
\begin{lemma}\label{l:high-al-t-r} 
 Let $V\in L^1_1$,  then there are $a_\pm(\R)$, %a_1\in L^1(\R)
 $b\in L^1(\R_-)$ such that for $|k|>k_0:=k_0(\Vert V\Vert_{L^1_1})>1$  \begin{align*}
%&\al_\pm(k)= 1-  \frac{\nu_0}{2ik}+  \frac{1}{2ik}\hat{a_\pm}(k) \\
&t(k)= %(1-\frac{\nu_0}{2ik}-  \frac{1}{2ik}\hat{\beta}(-2k) )^{-1} \\
          1+\sum_{n=1}^\infty (2ik)^{-n}(\nu_0+\hat{b}(k))^{n}\\
&r_\pm(k)= (-\nu_0+\hat{a_\pm}(k)) \sum_{n=1}^\infty (2ik)^{-n}(\nu_0+\hat{b}(k))^{n-1}\,,%\quad checked\; Apr.2, 08\\
%&r_-(k)=(-\nu_0+\hat{a}_-(k)) \sum_{n=1}^\infty (2ik)^{-n}(\nu_0+\hat{b}(k))^{n-1}
\end{align*}
where $\nu_0=\int V(t)dt$ and $k_0$ is a fixed constant depending on $\Vert V\Vert_{L^1_1}$ and %$\nu_0$ 
%\frac{c(\Vert V\Vert_{L^1_1})}{2}(|\nu_0|+\Vert V\Vert_{L^1_1})
 %$a_\pm\in L^1$ and %\hat{a}(k):=\hat{a_1}(2k)$  
 $\Vert a_\pm\Vert_1$, $\Vert b\Vert_1\le c(\Vert V\Vert_{L^1_1})$.      %:=\hat{\beta}(-2k)$ so that $b\in L^1(\R_-) $  %checked Jun 15 08
\end{lemma}

We will also need the relations between $m_+$ and $m_-$ \cite[Ch.2, %\S 3,
p.144]{DT79}.
\begin{lemma}\label{l:t-m+m-}  Let $V\in L^1_1$. 
\begin{align*}
&t(k)m_-(x,k)=e^{2ikx}r_+(k)m_+(x,k)+m_+(x,-k)\\ %r_1=r_+
&t(k)m_+(x,k)=e^{-2ikx}r_-(k)m_-(x,k)+m_-(x,-k) . %r_2=r_- \qquad (checked)
\end{align*} %where $t_+(k)=t_-(k)$.
\end{lemma}

\subsection{High energy cutoff for $\Phi_j(H_{ac})(x,y)$}\label{ss:weiLiy-hi-Phij} We are ready to prove (\ref{e:phi-dec-rhoj-zetaj}) for the high energy. %cutoff 
\begin{lemma}\label{l:high-ker-dec-mu}  Let $V\in L^1_1$ and $\Phi\in C^\iy([-1,1])$ as in (\ref{e:Phi-phij}). Then %in cases a, b, c 
there exists a finite measure $d\zeta_{high}$ %$\zeta_1$ 
in $\R_+\de+L^1$ %independent of $j$
such that for all $x,y$ and $j\ge j_0:=j_0(\Vert V\Vert_{L^1_1})$,
%$j_0$ being a sufficiently large number depending on $\Vert V\Vert_{L^1_1}$
\begin{align}\label{e:hi-Phij-zeta} 
|\big((1-\Phi_{j_0})\Phi_j\big)(H_{ac})(x,y)|
%\lesssim& \int \vert(\Psi_j^\vee-\Psi_{j_0}^\vee)(\pm x\pm y-u)\vert d\mu(u)  +
% \int \td{\rho}_0(\pm x\pm y-u) d\zeta(u)\\
%\int \td{\rho}_0(\pm x\pm y-u) d\zeta(u)\\
\le\sum_\pm(\rho_0+\rho_j)*d\zeta_{high}(\pm x\pm y) ,%+\rho_j%*d\zeta(\pm x\pm y)%\int \la \pm x\pm y-u \ra^{-2} d\zeta(u)\\
\end{align}
where %$\Vert \zeta\Vert_M$ depends on $\Vert V\Vert_{L^1_1}$ only, 
$j_0$ is a fixed %sufficient large 
number depending on $\Vert V\Vert_{L^1_1}$ only, %we can take 
$0\le\rho_0(x)\le c_N(1+|x|)^{-N}$, $0\le\rho_j(x)\le c_N2^{j/2}(1+2^{j/2}|x|)^{-N}$, $\forall N$.
\end{lemma}
\begin{proof}  
In the following we always assume 
$x>y$. The estimates for $x<y$ follow by symmetry. 
We divide the discussions into three cases. a)  $x>0,y<0$, 
b) $x>0, y>0$, and  c) $x<0,y<0$.

Let $\tilde{\psi}_j(k)=(1-\Psi_{j_0}(k))\Psi_j(k)$,
$\Psi_j(k)=\Phi_j(k^2)$. Let $j_0:= \max(2+[2\log_2k_0] , %checked Jul 3 08
2\log_2\Vert d\sigma\Vert_M)$, %-->2^{j_0/2}>k_0  k_0>0
$d\sigma= |\nu_0|\de + |b|$, where $k_0, b$ are the same as in Lemma \ref{l:high-al-t-r}.
%$\Vert \sigma\Vert\le 2\pi(|\nu_0|+\Vert b\Vert)$. 

\nd{\bf Case a.} $x>0$, $y<0$.  According to (\ref{e:ker-phi-m}) and Lemma \ref{l:high-al-t-r}, we have 
for  $j\ge j_0$, 
\begin{align*}
&2\pi\big((1-\Phi_{j_0})\Phi_j\big)(H_{ac})(x,y)\\%=\int \tilde{\psi_j}(k) t(k) m_+(x,k) m_-(y,k)e^{i(x-y)k} dk\\
 =&\sum_{n=0}^\infty  (1/2i)^n \int \tilde{\psi_j}(k) k^{-n}   (\nu_0 + \hat{b}(k))^{n}
  m_+(x,k) m_-(y,k)e^{i(x-y)k} dk\\ 
:=&\sum_{n=0}^\iy I_n(x,y). \end{align*} 
By Lemma \ref{l:m+m-rho},  %$V\in L^1_1$ and 
if $x>0$, $y<0$, 
 \begin{align*} 
 |m_+(x,\cdot)^\vee*m_-(y,\cdot)^\vee(u)|
\le%&  \de(u)+ \chi _{u<0}\rho^+(-u/2) +\chi _{u>0}\rho^-(-u/2) 
%+ \chi _{u<0}\rho^+(-u/2) * \chi _{u>0}\rho^-(-u/2) \\:=&
d\zeta_0:=c\de+\rho_1\in \R_+\de+L^1(\R_-).
\end{align*} 

If $n=0$,
\begin{align*}
&|I_0(x,y)|%&\int \tilde{\psi_j}(k) m_+(x,k) m_-(y,k)e^{i(x-y)k} dk\\
 =\frac{1}{4\pi^2}|\int \tilde{\psi_j}^\vee(x-y-u)   m_+(x,\cdot)^\vee* m_-(y,\cdot)^\vee(u) du|\\
  \le & \int |\tilde{\psi_j}^\vee(x-y-u)|  d\zeta_0(u)  ,
    \end{align*}
 where since $\Psi\in C^\iy_0$, we have
\begin{align*}
|\tilde{\psi_j}^\vee(x)|\le 2^{j_0/2}(1+2^{j_0/2}|x|)^{-N}+2^{j/2}(1+2^{j/2}|x|)^{-N}\end{align*}
 by writing  %with $j>j_0$
\begin{align}\label{e:psi-jj0} &\tilde{\psi_j}^\vee(\eta)= \Psi_j^\vee(\eta)-\Psi_{j_0}^\vee(\eta)\notag\\
=& 2^{j/2}\Psi^\vee(2^{j/2}\eta)-  2^{j_0/2} \Psi^\vee( 2^{j_0/2}\eta) .
\end{align}

For $n=1$, observe that
\begin{align} 
 &(\tilde{\psi_j}(k) k^{-1})^\vee (\xi)=
\frac{1}{2i} ( \int_\xi^\infty \tilde{\psi_j}^\vee(u)   du - 
\int_{-\infty}^\xi \tilde{\psi_j}^\vee(u)   du)\notag\\
=&-i \int_\xi^\infty \tilde{\psi_j}^\vee(u)   du= 
i \int_{-\infty}^\xi \tilde{\psi_j}^\vee(u)   du,\label{e:psij-k-xi}\end{align}
where $\int  \tilde{\psi_j}^\vee(u)du=2\pi \tilde{\psi_j}(0)=0$. %Jun 29 08
It is easy to see from (\ref{e:psij-k-xi}) and (\ref{e:psi-jj0}) that for each $N\in \N$ there exists a constant $c_N>0$ such that for all 
$j\ge j_0$,
\begin{align*}
\int_{\xi}^\infty |\tilde{\psi_j}^\vee(\eta)|   d\eta 
\le c_N (1+|\xi|)^{-N}\qquad \forall \xi . 
   \end{align*}
%and using \[  \Psi^\vee(\eta) \le c_N (1+|\eta|)^{-N-1}\]  we find if $\xi>0$
%\[ \int_\xi^\infty 2^{j/2} |\Psi^\vee(2^{j/2}\eta)| d\eta   
%=  \int_{2^{j/2}\xi}^\infty |\Psi^\vee(\eta)| d\eta\le \int_{2^{j_0/2}\xi}^\infty |\Psi^\vee(\eta)| d\eta   \qquad j>j_0\]   
%and if $\xi<0$ using the second equation of (\ref{e:psij-k-xi})
Thus   \begin{align*}
|I_1(x,y)|= &|\int \tilde{\psi_j}(k)k^{-1} (\nu_0+\hat{b}(k) ) m_+(x,k) m_-(y,k)e^{i(x-y)k} dk|\\
 \le & \int | (\tilde{\psi_j}k^{-1})^\vee(x-y-u) | (|\nu_0|\de+ |b|)*d\zeta_0(u) \\
  \le& c_N\int (1+|x-y-u| )^{-N}  d\zeta_1(u) ,
    \end{align*} 
where $d\zeta_1= (|\nu_0|\de+|b|) * d\zeta_0$ 
is in $\R_+\de+L^1(\R_-)$. %($\because \zeta\in L^1(\R_-)$) 

If $n\ge 2$, we have by integration by parts: for $j\ge j_0$, $N\ge 1$, 
\begin{align*} & |(1+\xi^N) \int \tilde{\psi_j}(k) k^{-n} e^{ik\xi} dk|\\ 
%=& |\int_{2^{(j_0-1)/2}\le |k|\le 2^{j/2}} \tilde{\psi_j}(k) k^{-n}  (1+(-i)^N\partial_k^N) [e^{ik\xi}] dk|\\ 
=& |\int_{2^{(j_0-1)/2}\le |k|\le 2^{j/2}} e^{ik\xi} (1+i^N\partial_k^N)\big[\tilde{\psi_j}(k) k^{-n}\big] dk|\\ 
\le & c_{j_0,N}  n^{N-1} 2^{-(j_0-1)n /2} \qquad  \forall \xi. %checked
\end{align*}
Hence, with $d\sigma= |\nu_0|\de + |b(u)|$, %Jul 17 08
\begin{align*}
&\sum_{n=2}^\iy |I_n(x,y)|\\%\sum_{n=2}^\infty  2^{-n} |\int \tilde{\psi_j}(k) k^{-n}   (\nu_0 + \hat{\beta}(-2k))^{n}
  %m_+(x,k) m_-(y,k)e^{ik(x-y)} dk |\\ 
%\le& \sum_{n=2}^\infty   2^{-j_0 n/2}  \int (1+|x-y-u|)^{-2}  \\
\le& c_N\sum_{n=2}^\infty n^{N-1}2^{-j_0 n/2}  2^{-n/2}\int (1+|x-y-u|)^{-N}  \overbrace{d\sigma*\cdots*d\sigma}^n
 *d\zeta_0(u)\\
%\le& c_N\sum_{n=2}^\infty n^{N-1} 2^{-j_0 n/2}  2^{-n/2}\int (1+|x-y-u|)^{-N}d\zeta_n(u)\\
\le& c_N\int (1+|x-y-u|)^{-N}d\td{\zeta}(u) ,
\end{align*}
where by our choice $j_0>2\log_2\Vert d\sigma\Vert_M-1$  so that  
\begin{equation*}
d\td{\zeta}:=\sum_{n=2}^\infty n^{N-1} 2^{-j_0 n/2}  2^{-n/2}\overbrace{d\sigma*\cdots*d\sigma}^n
 *d\zeta_0(u) \end{equation*} 
is a finite measure in $\R_+\de+L^1$. Combining the the above estimates for $I_n(x,y)$, $n=0,1$ and $\ge 2$, we thus establish 
(\ref{e:hi-Phij-zeta}) in Case a. %desired estimate
%noting that the measure of  $\overbrace%^{n-tuple} {d\sigma*\cdots*d\sigma}^{n-tuple}
%*d\zeta_0$  bd $\Vert d\sigma\Vert_M^n\cdot \Vert d\zeta_0 \Vert_M$
   
\nd {\bf Case b.} $x>y>0$.   
By (\ref{e:ker-phi-m}) and Lemma \ref{l:t-m+m-}
\begin{align*}
&2\pi  \big((1-\Phi_{j_0}) \Phi_j\big)(H_{ac})(x,y)\\%=\frac{1}{-2i}
%\int \Phi_j(k^2) t(k) m_+(x,k) m_-(y,k)e^{i(x-y)k} dk\\
%=&\int \td{\psi}_j(k)m_+(x,k) (e^{2iky} r_+(k)m_+(y,k) + m_+(y,-k) ) e^{i(x-y)k}dk\\
=& \int \td{\psi}_j^\vee(x+y-u)  (r_+(\cdot)m_+(x,\cdot) m_+(y,\cdot))^\vee(u) du  \\
+& \int \td{\psi}_j^\vee(x-y-u)  (m_+(x,\cdot) m_+(y,-\cdot))^\vee(u) du .
\end{align*}

Similar to Case a, using Lemma \ref{l:m+m-rho} and the formula for $r_+(k)$ in Lemma \ref{l:high-al-t-r}  we obtain
that there exists some finite measure  $d\zeta_2 %\zeta_2,\zeta_3
\in \R_+\delta+L^1$ so that for all
$x>0, y>0$ and $j\ge j_0$,%:= j_0(\Vert V\Vert_{L^1_1}) $,
\begin{align*}
 &| \big((1-\Phi_{j_0})\Phi_j\big)(H_{ac})(x,y) |\\
\le& %c_N {u<0} \sum_\pm( 1+ |x\pm y-u| )^{-N} d\zeta_0
\int |\tilde{\psi_j}^\vee(x-y-u)| d\zeta_2(u) +
c_N\sum_\pm \int (1+|x\pm y- u|)^{-N}  d\zeta_2(u) . 
\end{align*}%\edz{although $m_+(x,\cdot)^\vee* m_+(y,\cdot)^\vee$ supported in $u<0$, the support of $r_+^\vee$ is $\R$, so 
%$r^\vee*m_+(x,\cdot)^\vee* m_+(y,\cdot)^\vee$ supported in $\R$, i.e., cannot make $u<0$}

\nd {\bf Case c.} $0>x>y$. Similar to Case (b), we obtain that there exists some finite measure  $d\zeta_3\in \R_+\delta+L^1$ so that for all
 $0>x>y$ and $j\ge j_0$,
\begin{align*}
 &| \big((1-\Phi_{j_0})\Phi_j\big)(H_{ac})(x,y) |\\
%\lesssim &c_N\int%_{{\color{red}u<0}} 
%( 1+ |{\color{red}-x-y-u}| )^{-N} d\sigma_3\qquad{\bf checked Jun 11 08}\\
%+& \int |(\tilde{\psi_j})^\vee(x-y-u)| d\lam_3(u)\\
\le& c_N\int (1+|x+ y+ u|)^{-N}  d\zeta_3(u) 
 +\int |\tilde{\psi_j}^\vee(x-y-u)| d\zeta_3(u).
\end{align*}
\end{proof}

\subsection{Proof of Lemma \ref{l:high-al-t-r}} 
% Let $V\in L^1_1(\R)$,  then there are $a_\pm(\R)$, %a_1\in L^1(\R)
 %$\beta\in L^1(\R_+)$ such that for $|k|>k_0=\frac{c(\Vert V\Vert_{L^1_1})}{2}(|\nu_0|+\Vert V\Vert_{L^1_1})=:k_0(\Vert V\Vert_{L^1_1})>0$  \begin{align*}
%%&\al_\pm(k)= 1-  \frac{\nu_0}{2ik}+  \frac{1}{2ik}\hat{a_\pm}(k) \\
%&t(k)= (1-\frac{\nu_0}{2ik}-  \frac{1}{2ik} 
 %\hat{\beta}(-2k) )^{-1} \\
 %& \qquad =  1+\sum_{n=1}^\infty (2ik)^{-n}(\nu_0+\hat{b}(k))^{n}\\
%&r_+(k)= (-\nu_0+\hat{a}_+(k)) \sum_{n=1}^\infty (2ik)^{-n}(\nu_0+\hat{b}(k))^{n-1} \quad checked\; Apr.2, 08\\
%&r_-(k)=(-\nu_0+\hat{a}_-(k)) \sum_{n=1}^\infty (2ik)^{-n}(\nu_0+\hat{b}(k))^{n-1}\end{align*}
%where %$a_\pm\in L^1$ and %\hat{a}(k):=\hat{a_1}(2k)$, 
%$\hat{b}(k):=\hat{\beta}(-2k)$ so that $b\in L^1(\R_-) $  \end{lemma} 
%\begin{proof} 
By Lemma \ref{l:t-r-m-B}, if $|k|>k_0=k_0(\Vert V\Vert_{L^1_1})$ 
large enough, we have a geometric series expansion %$\nu_0:=\int V(t)dt$
\begin{align*}
& t(k)=(1-\frac{\nu_0}{2ik}- \frac{1}{2ik}\int V(t) dt\int_0^\infty B_+(t,y)e^{2iky}dy)^{-1}\\
%=&\sum_{n=0}^\infty   (\frac{\nu_0}{2ik}+ \frac{1}{2ik}\int_\R V(t) dt\int_0^\infty B_+(t,y)e^{2iky}dy)^{n}\\
=&\sum_{n=0}^\infty (2i k)^{-n}   (\nu_0+ \int V(t) dt\int_0^\infty B_+(t,y)e^{2iky} dy)^{n}\\
 %=&\sum_{n=0}^\infty  (1/2i)^n k^{-n}   (\nu_0 + \hat{\beta}(-2k) )^{n}
 =&\sum_{n=0}^\iy (2ik)^{-n}(\nu_0+\hat{b}(k) )^{n} ,
\end{align*}
where $\hat{b}(k)=\hat{\beta}(-2k)$ and %$\hat{\beta}(k)= \int_\R V(t) dt\int_0^\infty B_+(t,y)e^{-iky} dy$
$\beta(y)= \int V(t) \chi_{(0,\infty)}(y) B_+(t,y) dt$, which is in $L^1(\R_+)$ %(Jun 11 08)\\ 
 by Lemma \ref{l:B-L1} (a).
%\[\Vert\beta\Vert_1\le \int |V(t)|dt \int_0^\infty |B_+(t,y)| dy\le c\int (1+|t|)|V(t)|dt.\]  

Let \begin{equation}\label{e:al-rt}
\al_\pm(k)= (1+r_\pm(k))t(k)^{-1} , \end{equation}
then there exist $a_{\pm}\in L^1(\R)$ such that
\begin{align}\label{e:al-nu0-ak}
%&r_+(k)= \al_+(k)t(k)-1\\
\al_\pm(k)= 1-  \frac{\nu_0}{2ik}+  \frac{1}{2ik}\hat{a_\pm}(k)\, \qquad\forall k\neq 0 .
\end{align}
Indeed, similar to the way we deal with $t(k)$, write
\begin{align*}
&\al_+(k)=1+\frac{1}{2ik}\int (e^{-2ikt}-1)   V(t)m_-(t,k)dt\\ 
=&1+  \frac{1}{2ik}\int (e^{-2ikt}-1)   V(t)dt + 
\frac{1}{2ik}\int (e^{-2ikt}-1)   V(t)dt
\int^0_{-\infty}  B_-(t,y) e^{-2iky}dy \\%\quad Jun 11 08
=& 1-  \frac{\nu_0}{2ik}+  \frac{1}{2ik}  \hat{V}(2k) + 
%\frac{1}{2ik}[(\int_y^\inftyB_-(t,y-t)V(t)dt)^\wedge(2k)-(\chi_{(-\infty,0)}(y)\int_\RV(t)B_-(t,y)dt)^\wedge(2k)] checked
\frac{1}{2ik}[ \big( \chi_{(-\infty,\iy)}(y)\int^{\iy}_yV(t)B_-(t,y-t)dt\big)^\wedge(2k)\\
-&\big( \chi_{(-\infty,0)}(y)   \int V(t)
 B_-(t,y)dt \big)^\wedge(2k) ] .   \end{align*}     
 It is easy to see from Lemma \ref{l:B-L1} (a) that the last two functions of $y$ in the parentheses
are in $L^1$ if $V\in L^1_1$. %where \begin{align*}
%&\int_\R e^{-2ikt}V(t)dt\int^0_{-\infty}  B_-(t,y) e^{-2iky}dy \\
%=&\int_\R V(t)dt\int^0_{-\infty}  B_-(t,y) e^{-2ik(y+t)}dy \\
%=&\int_\R V(t)dt\int^{t}_{-\infty} B_-(t,y-t) e^{-2iky}dy \\
%=&\int^\iy_{-\infty}e^{-2iky} dy\int^{\iy}_y V(t) B_-(t,y-t) dt\\
%=&( \chi_{(-\infty,\iy)}(y) \int^{\iy}_yV(t)B_-(t,y-t)dt )^\wedge(2k) .  \quad checked Jun 11 08
 %\end{align*}
Thus (\ref{e:al-nu0-ak}) holds for $\al_+(k)$ with some $a_+\in L^1$, and so \begin{align*}
&r_+(k)= \al_+(k)t(k)-1\\
=&(-\nu_0+\hat{a_+}(k)) \sum_{n=1}^\infty (2ik)^{-n}(\nu_0+\hat{b}(k))^{n-1} . %Jul 02 08 
\end{align*} %where $a_+\in L^1$. 
Similarly we obtain the formulas for $\al_-(k)$ and $r_-(k)$. \hB
% Lemma \ref{l:t-r-m-B} \begin{align*}
%&\al_-(k)=(1+r_-(k))t(k)^{-1}\\
%=& 1+ \int_{-\infty}^\infty V(t)m_+(t,k)\frac{e^{2ikt}-1}{2ik} dt \quad checked\\%=& 1+ \int_\R \frac{e^{2ikt}-1}{2ik}   V(t)m_+(t,k)dt\\
%=& 1+ \int_\R \frac{e^{2ikt}-1}{2ik}   V(t)dt + 
%\int_\R\frac{e^{2ikt}-1}{2ik}   V(t)dt\int_0^\infty  B_+(t,y) e^{2iky}dy\\ 
%=&1-\frac{\nu_0}{2ik}+\frac{1}{2ik}\hat{V}(-2k)+ 
%\frac{1}{2ik}\int_\R V(t)dt
%\int_0^{\infty} B_+(t,y) (e^{2ik(y+t)}-e^{2iky})dy \\
%=& 1-  \frac{\nu_0}{2ik}+  \frac{1}{2ik}  \hat{V}(-2k) + %\frac{1}{2ik}[ (\int_y^\infty B_-(t,y-t) V(t)dt )^\wedge(2k)-&( \chi_{(-\infty,0)}(y)   \int_\R V(t)% B_-(t,y)dt )^\wedge(2k) ]   \qquad {\bf checked}
%\frac{1}{2ik}[ (\int^{y}_{-\infty}V(t)B_+(t,y-t)dt)^\wedge(-2k)\\
%-&( \chi_{(0,\infty)}(y)   \int_\R V(t)
 %B_+(t,y)dt )^\wedge(-2k) ]   \qquad checked Jun 11 08
%\end{align*} $\therefore$ \begin{align*}  &r_-(k)= \al_-(k)t(k)-1\\=&(-\nu_0+\hat{a}_-(k)) \sum_{n=1}^\infty (2ik)^{-n}(\nu_0+\hat{b}(k))^{n-1}  \end{align*}   
 
\section{Weighted $L^\iy$ estimates: Low energy}\label{s:wei-infty-low}  
In this section we prove (\ref{e:phi-dec-rhoj-zetaj})  for $\Phi_j(H_{ac})(x,y)$ for $j< j_0$, where
$j_0$ is taken to be the same number as in Lemma \ref{l:high-ker-dec-mu}. %as in Lemma \ref{l:low-ker-dec-mu}
Recall that $\Psi_j(k)=\Phi_j(k^2)=\Phi(2^{-j}k^2)$. The following lemma gives Fourier transform formulas of $t,r_\pm$ for the low energy.  
\begin{lemma}\label{l:low-t-r}  %$V\in L^1_1(\R)$.   
a) Let $V\in {L^1_1}$ and $\nu\neq 0$. %$t(0)=0$
 Then there exist $f_1,g_{1,\pm}\in L^1$ such that for all $j< j_0$,
\[(\Psi_j(k)t(k) )^\vee(u)= \Psi_j^\vee* f_1 (u)   %where $c=c(\Vert V\Vert_{L^1_3(\R)} )$ .
\] % for some $f_1\in L^1$.  
\[(\Psi_j(k)r_\pm(k) )^\vee(u)= \Psi_j^\vee* (g_{1,\pm}-\de) (u) ,  %where $c=c(\Vert V\Vert_{L^1_3(\R)} )$ .
\]  %for some $g_{1,\pm}\in L^1$.  
equivalently, %if $j\le j_0$, then
\begin{align*} 
&\Psi_j(k)t(k)= \Psi_j(k) \hat{f}_1(k) \\
&\Psi_j(k)r_\pm(k)= \Psi_j(k) (\hat{g}_{1,\pm}(k)-1) . \end{align*} 
b) Let $V\in L^1_2$ and $\nu=0$. Then there exist $f, g_\pm$ in $L^1$ such that 
\begin{align*} 
& t(k)= 1+\hat{f}(k)\\
& r_\pm(k)=\hat{g_\pm}(k) .     %\qquad \; if  \; V\in L^1_2
\end{align*}
\end{lemma} %\rk the estimates suffice to prove the low energy arising in every circumstance
We postponed the proof till Subsections \ref{ss:pf-tr-a} and \ref{ss:pf-tr-b}.
%\rk w/o using the zheng-harnack ineq  above  we might consider put hebish and hormander
%method together, namely for $j>0$ we might think of use Lemma \ref{l:j-horm-cond}  ... %(I had a look it seems work possibly...)
%for those $b_k$ with $\lam\sim \ell(I_k)$ $\lam=2^{-j/2}, j>0$   $b_k= (1-\Phi_{j_k}(H))b_k+ \Phi_{j_k}(H)b_k$
Combining Lemma \ref{l:low-t-r} and Lemma \ref{l:m+m-rho} we readily obtain the following lemma. 
\begin{lemma}\label{l:low-t-m+m-reson}   
a) Let $V\in L^1_1$ and $\nu\neq 0$. Then there exist positive functions $h_1$,
$h_2$ and $h_3$ in $L^1$ independent of $x,y$ such that\\ 
(i) $\forall x>0, y<0$,  \begin{align*} |(\Psi_j(k) t)^\vee *m_+(x,\cdot)^\vee*m_-(y,\cdot)^\vee(u)|
 \lesssim |\Psi_j^\vee|*(\de+h_1)(u)\end{align*}%\in c\de+L^1\\
(ii) $\forall x>0, y>0$, \begin{align*}
 |(\Psi_j(k)r_+)^\vee *m_+(x,\cdot)^\vee*m_+(y,\cdot)^\vee(u)|
 \lesssim |\Psi_j^\vee|*(\de+h_2)(u)\end{align*} %\in c\de+L^1\\
(iii) $\forall x<0, y<0$,  \begin{align*} |(\Psi_j(k)r_-)^\vee *m_-(x,\cdot)^\vee*m_-(y,\cdot)^\vee(u)|
\lesssim |\Psi_j^\vee|*(\de+h_3)(u) .%\in c\de+L^1 .
\end{align*}  %where $f_i\in L^1$
b) Let $V\in L^1_2$ and $\nu=0$. Then there exist positive functions 
$f_1$, $f_2$ and $f_3$ in $L^1$, independent of $x,y$, such that\\ %\begin{itemize} 
(i)  $\forall x>0, y<0$,  \begin{align*}
 |t^\vee *m_+(x,\cdot)^\vee*m_-(y,\cdot)^\vee(u)|\lesssim \de+f_1(u)\end{align*}
(ii)  $\forall x>0, y>0$,  \begin{align*} 
 |r_+^\vee *m_+(x,\cdot)^\vee*m_+(y,\cdot)^\vee(u)|\lesssim \de+f_2(u) \end{align*}
(iii) $\forall x<0, y<0$, \begin{align*} |r_-^\vee *m_-(x,\cdot)^\vee*m_-(y,\cdot)^\vee(u)|\lesssim \de+f _3(u).%\in c\de+L^1 
\end{align*} \end{lemma}%\edz{we intentionally omit the ineq for $m(x,\pm\cdot),m(y,\pm\cdot)$ since it is just similar ineq}  

Thus the estimate in (\ref{e:phi-dec-rhoj-zetaj}) for the low energy cutoff
follows from Lemma \ref{l:low-t-m+m-reson} by proceeding the way similar to (but much simpler than) the high energy case
in Subsection \ref{ss:weiLiy-hi-Phij}. 
%To deal with $\Phi_{j_0}(H_{ac})\Phi_{j_1}(H_{ac})=\Phi_{j}(H_{ac})$ let $j=\min( j_0,j_1)$
\begin{lemma}\label{l:low-ker-dec-mu}  Let $V\in L^1_1$ and $H_V$ has
no resonance at zero or $V\in L^1_2 $. 
Then %in cases a, b, c 
there exist a finite measure $d\zeta_{low}\in \R_+\de+L^1$ such that for all $j< j_0$ %=j_0(\Vert V\Vert_{L^1_1})$ 
\begin{align*} 
|\Phi_j(H_{ac})(x,y)|\le c\sum_\pm\int \vert \Psi_j^\vee(\pm x\pm y-u)\vert d\zeta_{low}(u) .  %+\int \la  x- y-u \ra^{-2} d\zeta(u)
\end{align*}  %where $c=c(\Vert V\Vert_{L^1_3(\R)} )$ 
\end{lemma}
The detail of the proof is straightforward and hence omitted. 

\subsection{Fourier transforms of $t(k)$, $r_\pm(k)$}
 Lemma \ref{l:low-t-r} %and Lemma \ref{l:t-r-reson}
tells that in the cases of $V\in L^1_1$, $\nu\neq 0$ and $V\in L^1_2$,
low energy cut-offs of $t(k), r_\pm(k)$ are the Fourier transforms of $L^1$ functions up to $c\de$.
We will show that this is true by Wiener's lemma  
\cite[Lemma 6.3]{Katz76}. 
\begin{lemma}\label{l:Wiener} (Wiener)   Let %$\hat{f}, \hat{h}$ be the Fourier transforms of 
$f,  h\in L^1(\R)$.   Suppose $\supp \,\hat{f}$ is compact and
$\hat{h}$ is nonzero on $\supp \,f$.   Then there exists some $g\in L^1(\R)$ such that 
$  \hat{f}= \hat{h} \hat{g} $ or $ \hat{f}/\hat{h}= \hat{g} $.
\end{lemma}

The following variant of Wiener's lemma %$n$-dimensional version of 
can be found in e.g. \cite[Ch.V, $\S$3]{St70}.   
\begin{lemma}\label{l:Wiener-Stein} Let $g\in L^1(\R)$ such that 
$\hat{g}(x)+1$ is nonzero for all $x$.  Then there exists a function $f\in L^1(\R)$ such that
\[  \hat{f}(x)+1= \frac{1}{\hat{g}(x)+1}\,. 
\]
\end{lemma}
Recall from \cite[Theorem 1]{DT79} that (i) if %resonance 
$\nu=0$, then $t(k)\neq 0$, $\forall k$. %\Im k\ge 0$.  
(ii) if %no resonance 
$\nu\neq 0$, then $t(0)=0$ but $t(k)\neq 0$, $\forall k\neq 0$ %\Im k\ge 0$, $k\neq 0$
(cf. also Lemma \ref{l:trm-asym-low}).

%{$\nu\neq 0$.  no resonance.}\label{ss:reson-free}
Since $W(k)=-2ik/t(k)$, by Lemma \ref{l:t-r-m-B} %or \cite[Theorem 2, p.159]{DT79}
\begin{align}
W(k)%=& -2ik(1-\frac{1}{2ik}\int_\R V(t)dt- \frac{1}{2ik}\int_0^\infty B(t)e^{2ikt} dt)\\
%=& -2ik (1-\frac{\nu}{2ik}- \int_\R V(t) (\frac{m_+(t,k)-m_+(t,0)}{2ik}) dt)\\  & ( \text{cf. p.149})\\
%=& -2ik (1- \int_\R V(t) \frac{m_+(t,k)}{2ik} dt)\notag\\ 
 =& -2ik (1-\frac{\nu}{2ik}- \int V(t) dt(\int_0^\infty B_+(t,y)\frac{e^{2iky}-1}{2ik} dy)\label{e:W-vb-1}\\
 =& -2ik (1-\frac{\nu_0}{2ik}- \int V(t) dt\int_0^\infty B_+(t,y)\frac{e^{2iky}}{2ik} dy) .\label{e:W-vb}
\end{align} %where we note that $\nu=\int V(t)m_+(t,0)dt=\int V(t)dt(1+\int_0^\iy B_+(t,y)dy)=\nu_0+\int V(t)dt\int_0^\iy B_+(t,y)dy$

\subsection{Proof of Lemma \ref{l:low-t-r} (a)}\label{ss:pf-tr-a} In this case          
$\nu=W(0)\neq 0$,  hence $W(k)\neq 0$, $\forall k$. 
 Write %Subsection \ref{ss:reson-free}
\begin{align}\label{e:Psi-t-k-W}
&\Psi_{j_0}(k)t(k)%=  \Psi_{j_0}(k) \frac{-2ik}{W(k)}
= \frac{-2ik\Psi_{j_0}(k)}{\chi(k)W(k)} ,
\end{align}
where we take $\chi\in C^\infty_0$ with $\chi(x)=1$ on $\supp\,\Psi_{j_0}$. %$[-2^{j_0/2}, 2^{j_0/2}  ]$. 
From (\ref{e:W-vb}) we have \begin{align*}  W(k)%=&-2ik(1-\frac{1}{2ik}\int_\R V(t)dt- \frac{1}{2ik}\int_0^\infty B(t)e^{2ikt} dt)\\
%=& -2ik (1-\frac{\nu}{2ik}- \int_\R V(t) (\frac{m_+(t,k)-m_+(t,0)}{2ik}) dt)\\ 
 = %\frac{-2ik \Psi_{j_0}(k)}\chi(k)
 -2ik+{\nu_0}+ \big( \chi_{(0,\infty)}(\cdot) \int V(t) B_+(t,\cdot)dt\big)^\vee(2k) ,
%=-2ik (1-\frac{\nu}{2ik}- \int_\R V(t) dt(\int_0^\infty B_+(t,y)\frac{e^{2iky}-1}{2ik} dy)
\end{align*}
where we note that in terms of Lemma \ref{l:B-L1} (a), 
 the function $y\mapsto   \chi_{(0,\infty)}(y)\int V(t) B_+(t,y)dt $ is in $L^1$ provided
 $V\in L^1_1$.
 %\footnote{similar to the proof in Lemma \ref{l:high-al-t-r} for $\al_\pm(k)$} 
 Thus we find that $\chi W$, which is nonzero on the support of $\Psi_{j_0}$,    
  is the Fourier transform of
 an $L^1$ function.   According to Wiener's lemma (Lemma \ref{l:Wiener}), 
 \begin{equation}\label{e:Psi-t-f}
 \Psi_{j_0}(k)t(k) = \hat{f}_1(k) \end{equation}
for some $f_1\in L^1$.    Hence for $j< j_0$
\begin{align*}%\label{e:}
(\Psi_j(k)t(k) )^\vee=c\Psi_j^\vee * (\Psi_{j_0}t(k))^\vee
=  \Psi_j^\vee* f_1 ,
\end{align*} where note that $\Phi_{j_0}(k)\equiv 1$ on support of $\Phi_j(k)$.

%To show the second equation in (b) 
Let  %$V\in L^1_1$. $\nu\neq 0$.  $t(0)=0$. $W(k)\neq 0$.   
$\al_\pm(k)= (1+r_\pm(k))t(k)^{-1}$, then %for $j\le j_0$
 \begin{align*} &\Psi_j(k) r_\pm(k)=  \Psi_j(k) \al_\pm(k)t(k)-\Psi_j(k) . \end{align*}
It is sufficient to deal with the first term.  %$\nu\neq 0$
By (\ref{e:al-nu0-ak}) %in the proof of Lemma \ref{l:high-al-t-r}, if $V\in L^1_1$
there exist $a_\pm\in L^1$ such that
%\begin{align*}
%\al_-(k)=&(1+r_-(k))t(k)^{-1}
%= 1+ \int_\R \frac{e^{2ikt}-1}{2ik}   V(t)dt \\
%+& \int_\R\frac{e^{2ikt}-1}{2ik}   V(t)dt\int_0^\infty  B_+(t,y) e^{2iky}dy \end{align*}
\begin{align*}
2ik\al_\pm(k)% 2ik+\int_\R (e^{2ikt}-1)   V(t)dt \\+&\int_\R (e^{2ikt}-1)V(t)dt\int_0^\infty  B_+(t,y) e^{2iky}dy \\
%=&  2ik- \nu_0+\hat{V}(-2k) \\+& \int_\R e^{2iky} dy
%( \int^y_{-\infty} V(t)B_+(t,y-t) dt) -\int_0^\infty e^{2iky} dy(\int_\R V(t) B_+(t,y) dt)\\
=&  2ik- \nu_0+%  \hat{V}(-2k) + 
\hat{a}_\pm(k) . \end{align*} %for some $a_\pm\in L^1$.  %$(1+|t|)V(t)\in L^1$
Thus $\Psi_{j_0}(k)(-2ik)\al_\pm(k) = \hat{g_{0,\pm}}(k) $ for some $g_{0,\pm}\in L^1$.
We have for $j<j_0$,   
\begin{align*}
(\Psi_j(k) \al_\pm(k)t(k))^\vee  %\frac{-2ik}{W(k)}\\
=&  c\Psi_{j}^\vee*(\Psi_{j_0}(k)\frac{t(k)}{-2ik})^\vee *(\Psi_{j_0}(k)(-2ik)\al_\pm(k) )^\vee \\
=& c\Psi_{j}^\vee*(\frac{\Psi_{j_0}(k)}{W(k)})^\vee*g_{0,\pm} \\
=& \Psi_{j}^\vee*g_{1,\pm} , \qquad g_{1,\pm}=c{f_0}* g_{0,\pm}
\end{align*} %\edz{$(f*g)^\wedge=\hat{f}\hat{g}$, $(uv)^\vee=\frac{1}{2\pi}u^\vee*v^\vee$}
where  in view of (\ref{e:Psi-t-k-W}), the same way as showing (\ref{e:Psi-t-f}) we see that $\frac{\Psi_{j_0}(k)}{W(k)}=\hat{f_0}$
for some $f_0\in L^1$. This proves that for $j<j_0$
\[ \Psi_j(k) r_\pm(k)= \Psi_{j}(k)({\hat{g_{1,\pm}}(k)}-1) .\]   \hB 

\subsection{Proof of Lemma \ref{l:low-t-r} (b)}\label{ss:pf-tr-b} First we observe the following formula when $\nu=0$,
\begin{align}\label{e:t-inv-VB}
t(k)^{-1}-1 =&%-\int_\R V(t)dt \int_0^\infty B_+(t,y) \frac{e^{2iky}-1}{2ik} dy 
-\int V(t)dt \int_0^\infty (\int_\xi^\infty B_+(t,y)dy)  e^{2ik\xi} d\xi\\
=&-\big(  \chi_{(0,\infty)}(\xi)\int V(t)dt \int_\xi^\infty B_+(t,y)dy\big)^\vee(2k) .\notag 
\end{align}
Indeed, %note that resonance means $\int_\R V(t)(1+\int_0^\infty B_+(t,y) dy) dt =0$ 
since $\nu=0$, we have by (\ref{e:W-vb-1})%Lemma \ref{l:t-r-m-B} %$m_+(x,k)=1+ \int_0^\infty B(t,y) e^{2iky}dy $   
\begin{align*}
 t(k)^{-1}%=& 1- \int V(t)\frac{m_+(t,k)-m_+(t,0) }{2ik}dt\\
=1-\int V(t)dt \int_0^\infty B_+(t,y) \frac{e^{2iky}-1}{2ik} dy .
\end{align*}
Then %Integrating by parts e.g. 
 (\ref{e:t-inv-VB}) %the equation 
follows by using $\frac{e^{2iky}-1}{2ik} =\int_0^y e^{2i k\xi}d\xi$ and %exchanging order of integrals. %
Fubini theorem.

Since $V\in L^1_2$, Lemma  \ref{l:B-L1} (b) implies that the function given by\\
$\xi\mapsto  \chi_{(0,\infty)}(\xi)\int V(t)dt \int_\xi^\infty B_+(t,y)dy$
belongs to $L^1$. Hence $t(k)\inv-1=\hat{g}_0(k)$ for some $g_0\in L^1$.

Now by Lemma \ref{l:Wiener-Stein} there exists $h\in L^1$ (evidently $1+\hat{g}_0(k)=t(k)\inv\neq 0$, $\forall k$)
so that
\[ t(k)=\frac{1}{1+\hat{g}_0(k)}=1+\hat{h}(k).
\]

A similar argument shows that there exists some $\om_\pm\in L^1$ so that 
$\al_\pm(k)=(1+r_\pm(k))t(k)^{-1}= 1+\hat{\om_\pm}(k)$
by applying Lemma \ref{l:B-L1}.  %\edz{\crr For $\al(k)$ we can have this only if $V\in L^1_2$ actually we have a form with an extra $k$ term if only $V\in L^1_1$}
%According to a), if $\nu=0$ %Wiener's lemma
%\[t(k)=1+ \hat{h}(k), \quad h\in L^1 .\] 
Therefore $r_\pm= \al_\pm t-1%= \al(k)+ \al_\pm(k)\hat{h}(k)-1
= \hat{\om_\pm}+\hat{h}+\hat{\om_\pm}\hat{h} $ are in $(L^1)^\wedge$.  \hB 
%\edz{[DF06] has a pf of the low energy cutoff belonging to $(L^1)^\wedge$ by introducing 
%$n(x,k)=m(x,k)-m(x,0)$, bound of $\Vert n\Vert_1\sim\Vert y\pa_yB(x,y)\Vert_1$ and then apply Wiener
%(not Stein). Here we use Stein to give a simple and stronger version for F-transf of $t(k)$} 

\subsection{Marchenko equation}\label{e:B-march}  From Lemma \ref{l:t-r-m-B} we know that for each $x$, $B_\pm(x, y)$ 
are the Fourier transforms of $m_\pm(x,\pm k)-1$. %It is well-known that 
They are %well defined 
real-valued, supported in $\R_\pm$ and belong to $L^2(\R_\pm)$  %so that $m_\pm(x,k)-1$ both belong to the Hardy space $H^{2+}$ 
\cite{DT79,W99}. %$B_-(x,y)$ belongs to $L^2(\R)$ supported in $\R_-$, $m_-(x,k)-1\in H^{2+}$
%The formula (2.24) in [W99] is  correct  for $e^{\pm 2iky}$, in consistent with [DT79]}
%which play an important role in the kernel estimates of $\phi_j(H)$, especially for low energy analysis
Moreover, %\cite{DT79} 
$B_\pm(x,y)$ satisfy the {\em Marchenko equations}
\begin{equation}\label{e:B+marchenko}
B_+(x,y)= \int_{x+y}^\infty V(t)dt+ \int_0^y dz \int_{t=x+y-z}^\infty V(t) B_+(t,z) dt 
\end{equation}
\begin{equation}\label{e:B-marchenko}
B_-(x,y)= \int^{x+y}_{-\infty} V(t)dt+ \int^0_y dz \int^{t=x+y-z}_{-\infty} V(t) B_-(t,z) dt .
\end{equation} %note $B_-$ supported in $\R_-$

The following lemma is mainly on certain weighted $L^1$ inequalities for $B_\pm$, which contributes to several %main 
kernel estimates as we have seen. 
 %as stated in parts (a) and (b) of Lemma \ref{l:B-L1}.
\begin{lemma}\label{l:B-L1} 
(a) If $V\in L^1_1$,  then there exists $c=c(\Vert V \Vert_{L^1_1} )$ so that 
for all $x\in \R$ 
\begin{equation}
\int_{-\iy}^\iy | B_\pm(x,y) | dy\le c (1+\max(0, \mp x)) .
\end{equation}
%\begin{equation}
%\int_\R | B_-(s,y) | dy\le c (1+\max(0, s))  \end{equation}
(b) If $V\in L^1_2$,  then there exists $c=c(\Vert V \Vert_{L^1_2} )$ so that 
for all $x\in \R$ 
\begin{equation}
\int_{-\iy}^\iy |y|  | B_\pm(x,y) | dy\le c (1+\max(0, \mp x))^2 . 
\end{equation}
%where $c$
(c) Let $V\in L^1_1$, then for all $x,y\in\R$
\begin{align}\label{e:B+-rho}
&|B_\pm(x,y)| \le e^{\ga^\pm(x)} \rho^\pm(x+y) ,
%& |B_-(x,y)| \le e^{\ga^-(x)} \rho^-(x+y) \quad \forall y<0, x\in \R  \notag
\end{align}
where $\ga^+(x)=\int^{\infty}_x (t-x)|V(t)|dt$, %$\rho^+(x)=\int^{\infty}_x |V(t)|dt $,
$\ga^-(x)=\int_{-\infty}^x (x-t)|V(t)|dt$, $\rho^\pm$ are as in Lemma $\ref{l:m+m-rho}$. %(x)=\int_{-\infty}^x |V(t)|dt $ DT79 DF06
\end{lemma}
 Estimate (c) is known, see e.g. \cite{DT79} or \cite{GS04}. 
The estimates in  (a), (b) were obtained in \cite[Lemma 3.2, Lemma 3.3]{DF06}
using Gronwall's inequality. 
See also \cite[Lemma 4.5]{OZ08} for a generalized version
of these inequalities for $n\in \N_0=\{0,1,2,\dots\}$
\begin{equation}
\int_{-\iy}^\iy |y|^n  | B_\pm(x,y) | dy\le c (1+\max(0, \mp x))^{n+1} ,
%\int_\R |y|^n  | \pa^kB_+(s,y) | dy\le c (1+\max(0, -s))^{n}, \quad n\ge 1  \pa=\pa_y or \pa_s
\end{equation}
$c=c(\Vert V\Vert_{L^1_{n+1}})$, 
which %It seems that the most efficient way to deal with $B_\pm(s,k)$ %$s<0$
follows from direct iterations of (\ref{e:B+marchenko}) and (\ref{e:B-marchenko}).%the Marchenko equation

\subsection{Modified Jost functions} 
Let $h(x,k)=\frac{e^{2ikx}-1}{2ik}$. %=\int_0^x e^{2ikt}dt$  
It is well-known that $m_\pm(x,k)$ satisfy the equations%improve an estimate of \cite[p.134]{DT79}
\begin{align}
 &m_+(x,k)=1+\int_x^\infty h(t-x,k)V(t)m_+(t,k)dt\label{e:m+hVm+}\\
 &m_-(x,k)=1+\int^x_{-\infty} h(x-t,k)V(t)m_-(t,k)dt .\label{e:m-hVm-}
\end{align}

\begin{lemma}\label{l:der-m+} Let $V\in L^1_1$. %For $k\in \R$
Then  \begin{align*}
&|m_\pm(x,k)| \le c(1+\max(0,\mp x))\\ 
&|\dot{m}_\pm(x,k)| \le
%\begin{cases}  \frac{\Vert V\Vert_{L^1_1} |k|^{-1} %(\Vert V\Vert_{L^1_2} + 
 c\frac{1+\max(0,\mp x)}{|k|} \qquad \forall k\neq 0 , %e^{\Vert V\Vert_{L^1_1}} & x< 0\\
 %c_2(1+x^2)  \end{cases}
\end{align*}
where $c=c(\Vert V\Vert_{L^1_1} )$. 
\end{lemma}

\begin{lemma}\label{l:trm-asym-low}  (a) Let $V\in L^1_1$. %and $\nu\neq 0$% or $V\in L^1_2$ $\nu=0$   
Then:  %\begin{align*}%\label{e:tr-high-asym}
(i) $|t(k)|\le 1, \,|r_\pm(k)|\le 1$. \\
(ii) $\dot{t}(k)=O(1/k),\, \dot{r}_\pm(k)=O(1/k)$ as $|k|\to\iy$.
(iii) If $\nu\neq 0$, then $t(k)=O(k)$ as $k\to 0$.  

(b) Let $V\in L^1_1$ and $\nu\neq 0$ or $V\in L^1_2$. %(asymptotics   $k\to 0$) 
Then \begin{align*} %t(k)=O(k) \nu\neq 0\\
\dot{t}(k)=O(1/k),\, \dot{r}_\pm(k)=%\begin{cases}O(1)=O(1/k) & L^1_1, \nu\neq 0\\
O(1/k) \,\text{as $k\to 0$}.%& L^1_2, \nu=0\end{cases}\\
\end{align*}\end{lemma}  The asymptotics %for $m_\pm(x,k), t(k),r_\pm(k)$ 
in Lemma \ref{l:der-m+} and
Lemma \ref{l:trm-asym-low} (a) are known, see \cite[\mbox{Lemma 1,} p.130]{DT79} or \cite{W99}. 
We will give the proof of Lemma \ref{l:trm-asym-low} (b) below. 
%\edz{[W00] also get estimate in the low energy; however there was a mistake in using the formula (2.34)
%for $\dot{m}(x.k)$ to prove for $\dot{t}(k)$}

\subsection{Proof of Lemma \ref{l:trm-asym-low} (b)} 
{\bf A}. Let $V\in L^1_1$ and $\nu\neq 0$. By Lemma \ref{l:t-r-m-B} we have
%the properties for $r_\pm$ are from  \cite[Theorem 1]{DT79}
 \begin{align*} 
t(k)^{-1}= 1-\frac{1}{2ik}\int V(t)m_+(t,k)dt .\end{align*}
 
Taking derivative in $k$ and applying Lemma \ref{l:der-m+}  give 
\begin{align*} 
|\partial_k (t(k)^{-1})| \le 
c/k^2\qquad \forall k\neq 0,
\end{align*}
thus, 
\begin{align}\label{e:der-t--no-reson} 
\dot{t}(k)=-t(k)^2\partial_k (t(k)^{-1}) =\begin{cases}O(1)& |k|<1\\ 
O(1/k^2)&|k|\ge 1 ,
\end{cases}
\end{align}
where we used $t(k)=O(k)$, $k\to 0$ if $\nu\neq 0$, by Lemma \ref{l:trm-asym-low} (a).
From  (\ref{e:al-rt}) and Lemma \ref{l:t-r-m-B} we see
\begin{align*}%\label{e:r+-thVm} 
{r}_\pm(k)= {t}(k)\big(1+ \int h(\mp t,k)V(t)m_\pm(t,k)dt\big)-1 .\end{align*}

Now the estimate \[ |\dot{r}_\pm(k)|\le c/|k|  \] can be established
by using (\ref{e:der-t--no-reson}), the estimates  
\begin{align*}%\label{e:h-derh-asym}
& |h(t,k)|\le \min (|t|, 1/ |k|) \quad   \\ 
& |\dot{h}(t,k)| 
\le 2\frac{|t| }{|k|} \,,  
\end{align*}   
%where \begin{align*}
%& \dot{h}(t,k)=\int_0^t 2iue^{2iku}du=\frac{1}{k}\int_0^t 2iud_u(e^{2iku})\\ 
%=&\frac{1}{k}\left(\int_0^t ud_u(e^{2iku})\right)=\frac{1}{k}\left( ue^{2iku}\vert_0^t-\int_0^t e^{2iku}du\right)\\
%\le& 2\frac{|t| }{|k|}, k\to 0    \end{align*}     $t\in\R$ we use for $x<0$ the formula in 
 Lemma \ref{l:der-m+} and Lemma \ref{l:t-m+m-}. %A key observation is {\color{orange}$\lozenge$}
%\[ t(k)m_+(x,k)=e^{-2ikx}r_-(k)m_-(x,k)+m_-(x,-k) .  %\qquad {\color{red}\Diamond} \]
%\begin{align*}
%&t(k)  \int_0^\infty \dot{h}(t,k)V(t)m_+(t,k)dt+ t(k)  \int_{-\infty}^0 \dot{h}(t,k)V(t)m_+(t,k)dt\\ :=&\int^+ +\int^- \end{align*}
%$\int^+\le \frac{|t(k)|}{|k|} \Vert tV\Vert_{L^1}$ ok.    It remains to deal with 

{\bf B}. Let $V\in L^1_2$ and $\nu=0$. By (\ref{e:t-inv-VB}), integrating by parts  %when $\nu=0$
%\begin{equation}
%t(k)^{-1}-1%-\int_\R V(t)dt \int_0^\infty B(t,y) \frac{e^{2iky}-1}{2ik} dy 
%=  -\int_\R V(t)dt \int_0^\infty (\int_\xi^\infty B(t,\eta)d\eta)  e^{2ik\xi} d\xi 
%\end{equation}
we have 
\begin{align*}
& k\pa_k(t(k)\inv)=-\int V(t)dt \int_0^\infty \big(\int_\xi^\infty B_+(t,\eta)d\eta\big)\xi d_\xi(e^{2ik\xi})\\  
=&\int V(t)dt %\bigg( e^{2ik\xi} \xi\int_\xi^\infty B(t,\eta)d\eta\vert_0^\iy\\
\int_0^\iy e^{2ik\xi} \big(\int_\xi^\infty B_+(t,\eta)d\eta-\xi B_+(t,\xi)\big)d\xi , %\bigg)  
\end{align*} 
where note that $|\xi\int_\xi^\infty B_+(t,\eta)d\eta|\le\int_\xi^\infty |\eta B_+(t,\eta)|d\eta\to 0$ as $\xi\to \iy$
by virtue of Lemma \ref{l:B-L1} (b).
%bec $\int_0^\infty |\eta B_+(t,\eta)|d\eta\le c(\Vert V\Vert_{L^1_2})(1+|t|)^2$.
 Then Fubini theorem and again Lemma \ref{l:B-L1} (b) give \begin{align*}
& |k\pa_k(t(k)\inv)| %\le\int_\R |V(t)|dt \bigg( %e^{2ik\xi} \xi\int_\xi^\infty B(t,\eta)d\eta\vert_0^\iy\\
%\int_0^\iy\big\vert\int_\xi^\infty B(t,\eta)d\eta-\xi B(t,\xi)\big\vert d\xi\bigg) \\  
\le\int |V(t)|dt \bigg(
\int_0^\iy\big(\int_\xi^\infty |B_+(t,\eta)|d\eta+\xi |B_+(t,\xi)|\big) d\xi\bigg)\\
\le& c\int (1+|t|)^2|V(t)|dt .   %\text{provided $V\in L^1_2$}
\end{align*}
Thus \begin{align}\label{e:de-t-reson-low} 
\dot{t}(k)=-t(k)^2\partial_k (t(k)^{-1}) =O(1/k) . %\qquad\forall k.  %(L^1_2, \nu= 0)
\end{align}
Finally, the estimate $\dot{r}_\pm(k)=O(1/k)$ follows from those routine asymptotics for $\dot{h}(t,k), \dot{m}_\pm(x,k)$ and 
(\ref{e:de-t-reson-low}). \hB

\begin{remark} The proof in Part A actually has shown the asymptotics for $t(k), r_\pm(k)$
for both $k\to 0$ and $|k|\to \iy$, for the latter we only require $V\in L^1_1$. \end{remark}

\section{Weighted $L^2$ estimates}\label{s:wei-L2-lowhi}  
We will prove (\ref{e:wei-L2-ineq})  for $H_{ac}$  with $s=1$ and $s=0$
(Lemma \ref{l:wei-L2-mu-lo} and Lemma \ref{l:j-ker-L2}). Then 
the case $1/2<s<1$ follows via interpolation. 
The case $s=1$ requires certain %refined 
improved asymptotics for $k$-derivatives of $m_\pm(x,k)$ $t(k), r_\pm(k)$ than
 \cite[p.134]{DT79}. 
In the most difficult (and subtle) case (Case c) we use the Volterra type 
expansions for $\dot{m}_\pm(x,k)\;$\footnote{We will follow the convention that $\dot{f}(k)=\pa_kf(k)$.} in order to deal with the 
inconsistency of the weight and distorted phase.

In the following we use the abbreviations $H^s=W^s_2$, $\dot{H}^s=\dot{W}^s_2$,
and so, %$H^s([\frac14,1])=W^s_2([\frac14,1])$. 
\begin{align}
 &\Vert f\Vert_{H^s(\R)}=\Vert (1+|\xi|^2)^{s/2}\hat{f}\Vert_2 \notag\\
 &\Vert f\Vert_{\dot{H}^s(\R)}=\Vert |\xi|^s\hat{f}\Vert_{2}\approx \Vert (-\De)^sf\Vert_2 \,. \label{e:Hs-homog}\end{align}
\begin{lemma}\label{l:wei-L2-mu-lo} Let $V\in L^1_1$, $\nu\neq 0$ or $V\in L^1_2$. 
Then for all $y$, $j\in\Z$ and $\phi\in H^1([\frac14,1])$, 
\begin{align}\label{e:w-L2} 
\Vert (x-y) \phi_j(H_{ac})(x,y)\Vert_{L^2_x} \le c2^{-j/4}\Vert \phi\Vert_{H^1([\frac14,1])} \,. 
\end{align} 
 \end{lemma}
 \begin{proof} Let $\psi_j(k):=\phi_j(k^2)= \phi(2^{-j}k^2)$. %$j\in\Z$ 
 By symmetry we will only need to show (\ref{e:w-L2}) in the following three cases:
\begin{align*}
(\ref{e:w-L2}a)\;  \forall y<0,\;\; &\Vert \chi_{\{x>0\}}(x-y) \phi_j(H)(x,y)\Vert_{L^2_x}\le c2^{-j/4}\Vert \phi\Vert_{H^1([\frac14,1])} \\ 
%\label{e:x+y-wL2} \\
(\ref{e:w-L2}b) \;\forall y>0,\;\; &\Vert \chi_{\{x>y\}}(x-y) \phi_j(H)(x,y)\Vert_{L^2_x}\le c2^{-j/4}\Vert \phi\Vert_{H^1([\frac14,1])} \\ 
%\label{e:x+y+wL2} \\
(\ref{e:w-L2}c) \;\forall y<0,\;\; &\Vert \chi_{\{y<x<0\}}(x-y) \phi_j(H)(x,y)\Vert_{L^2_x}\le c2^{-j/4}\Vert \phi\Vert_{H^1([\frac14,1])} \,. 
%\label{e:x-y-wL2} 
\end{align*}

\nd{\bf Case a.}  $x>0$, $y<0$.   Using the formula in Lemma \ref{l:t-r-m-B}
\begin{align}\label{e:m+B+} 
m_+ (x,k)= 1+\int_0^\infty B_+(x,u)e^{2iku}du \,,\end{align} %(Lemma \ref{l:t-r-m-B}) %\end{align}  $h(t,k)= \frac{e^{2ikt}-1}{2ik}$
we write \begin{align*}
&2\pi \phi_j(H_{ac})(x,y)   %\int \psi_j(k) t(k) m_+(x,k) m_-(y,k)e^{i(x-y)k} dk \\
=  \int%_{|k|\sim 2^{j/2}}
\psi_j(k)t(k) m_-(y,k)e^{i(x-y)k} dk\\
+& \int%_{|k|\sim 2^{j/2}}
\psi_j(k)t(k)m_-(y,k)e^{i(x-y)k} dk\int_0^{\infty} B_+(x,u)e^{2iku}du\\
:=& I^a_1(x,y)+I^a_2(x,y) .
\end{align*}
In view of (\ref{e:Hs-homog}) %Plancherel formula for Fourier transform gives
we have \begin{align*}
&  \Vert (x-y)I^a_1(x,y)\Vert_{L^2_x}%=& \Vert  \int \partial_k^s[ \psi_j(k) t(k) m_-(y,k)] e^{i(x-y)k} dk \Vert_{L^2_x}\\
= \Vert \psi_j(k) t(k) m_-(y,k) \Vert_{\dot{H}^1_k}
%\le&\Vert (\psi_j(k))'O(1) \Vert_{L^2_k}+\Vert O(1) (\chi_j(k)t(k) m_-(y,k))' \Vert_{L^2_k} \\ 
%\le& c\Vert \psi_j\Vert_{\dot{H}^s_{[1/4,1]}}\Vert \chi_j(k)t(k)m_-(y,k)\Vert_{\dot{H}^s_{[1/4,1]}}\quad \because s>1/2\\
\le c2^{-j/4}\Vert \phi\Vert_{H^1([\frac14,1])} , 
\end{align*}
where we have used  the following estimates %for $k\to 0$ ($j\to -\infty$) If $V\in L^1_2$ or $V\in L^1_1$, $\nu\neq 0$ 
by Lemma \ref{l:der-m+} and \mbox{Lemma \ref{l:trm-asym-low}:} For $i=0,1$,
\[  \begin{cases} %|\psi_j(k)|\le& c\Vert \phi\Vert_{H^1_{[1/4,1]}}\\
\Vert \pa^i_k\psi_j\Vert_{L^2}\le &c2^{-\frac{j}{2}(i-\frac12)}\Vert \phi\Vert_{H^1([\frac14,1])}\\
%\partial_k \psi_j (k)=& O(1/k)\\
%t(k)=&O(1)  \\
\pa_k^i{t}(k)=%& -t(k)^2 \partial_k(t(k)^{-1} )\\
%=O(k^2)O(1/k^2)=
& O(1/k^i) \\
%m_-(y,k)\le&c   \qquad \forall y<0\\ 
\pa_k^i{m}_-(y,k)=&O(1/k^i) , \qquad \forall y<0.
%m_+(t,k)\le&c(1+\max(0,-t))  \quad t>0 
%\dot{m}_+(t,k)\le&c\frac{1+\max(0,-t)}{|k|} 
\end{cases}
\]  
Applying Minkowski inequality and Lemma \ref{l:B-L1} (c) for $x>0$, %$B_+(x,u)\le \rho^+(u)$, $\rho^+(u)=\int_u^\infty |V(t)| dt$, 
we obtain by (\ref{e:Hs-homog}) that for each $y<0$,  %  $\Vert f\Vert_{\dot{H}^s}=\Vert \xi^s\hat{f}\Vert_{L^2}$
%or by integration by parts (and use weak derivative)
\begin{align*}
& \Vert \chi_{\{x>0\}}(x-y) I^a_2(x,y)\Vert_{L^2_x}\\
=& \Vert \int_0^{\infty} B_+(x,u)du\,  (x-y)\int \psi_j(k)t(k)m_-(y,k)e^{i(x-y+2u)k} dk \Vert_{L^2_{\{x>0\}}}\\
\le& \int_0^{\infty} \rho^+(u)du\,\Vert (x-y+2u)\int \psi_j(k)t(k)m_-(y,k)e^{i(x-y+2u)k} dk \Vert_{L^2_{\{x>0\}}}\\ 
 %(& |x-y|\le |x-y+2u|, u>0)\\
\le& \int_0^{\infty} \rho^+(u)du\,\Vert \psi_j(k)t(k)m_-(y,k)\Vert_{\dot{H}^1_k}
%\le& \Vert  \int_0^{\infty} \rho(x+u) (\partial_k[\mu_j(\cdot)t(\cdot)m_-(y,\cdot)] )^\vee (x-y+2u) du\Vert_2\\
%\le& \int_0^{\infty} \rho(u) \Vert(\partial_k[\mu_j(\cdot)t(\cdot)m_-(y,\cdot)] )^\vee(x)\Vert_{L^2_x} du\\
%=& \int_0^{\infty} \rho(u) \Vert\partial_k[\psi_j(\cdot)t(\cdot)m_-(y,\cdot)] \Vert_{L^2_k} du\\
\le c 2^{-j/4}\Vert \phi\Vert_{H^1([\frac14,1])}.
\end{align*} 
So combining the estimates for $I^a_1$ and $I^a_2$ gives (\ref{e:w-L2}a). 

\nd{\bf Case b.}   $x>y>0$. By Lemma \ref{l:t-m+m-} we write
\begin{align*}
&2\pi  \phi_j(H_{ac})(x,y)%\int \psi_j(k) t(k) m_+(x,k) m_-(y,k)e^{i(x-y)k} dk\\
%=  \int_{|k|\sim\lam^{-1}}\psi_j(k)t(k) m_+(x,y) 
%[ e^{2iky } r_+(k) m_+(y,k) +m_+(y,-k) ] e^{i(x-y)k} dk\\
= \int\psi_j(k)r_+(k) m_+(x,k)m_+(y,k)e^{i(x+y)k}dk \\
+&  \int  \psi_j(k) m_+(x,k) m_+(y,-k) e^{i(x-y)k}dk .%:= I^b_1(x,y)+I^b_2(x,y) .
\end{align*}

By (\ref{e:m+B+}) %Lemma \ref{l:t-r-m-B} 
we see that (\ref{e:w-L2}b) can be proved as in Case a by applying Lemma \ref{l:B-L1} (c), Lemma \ref{l:der-m+} and Lemma \ref{l:trm-asym-low}.

\nd{\bf Case c.}  $y<x<0$. Using Lemma \ref{l:t-m+m-} we write
\begin{align*}
&2\pi  \phi_j(H_{ac})(x,y)%\int \psi_j(k) t(k) m_+(x,k) m_-(y,k)e^{i(x-y)k} dk\\
%=  \int\psi_j(k) m_-(y,k)[ e^{-2ikx } r_-(k) m_-(x,k) +m_-(x,-k) ] e^{i(x-y)k} dk\\
=\int\psi_j(k)r_-(k) m_-(x,k)m_-(y,k)e^{-i(x+y)k}dk \\
+&  \int  \psi_j(k) m_-(x,-k) m_-(y,k) e^{i(x-y)k}dk
:= I^c_1(x,y)+I^c_2(x,y) .
\end{align*}

The term $I^c_1$ can be dealt with in a way similar to Case a or b. %({\bf Marchenko}) %$y<x<0$
We can estimate $\Vert \chi_{\{y<x<0\}}(x-y)I^c_1(x,y)\Vert_2$ by writing
\[m_-(x,k)=1+\int_{-\infty}^0 B_-(x,u)e^{-2iku}du\]
(cf. Lemma \ref{l:t-r-m-B}, \cite[%\S 2, 
p.137]{DT79} %(the fact that $m_\pm-1\in H^2_+$)
or \cite{W99})
and using Lemma \ref{l:B-L1} (c) and  the estimates for $m_-(y,k)$, $r_-(k)$ in Lemmas \ref{l:der-m+} and \ref{l:trm-asym-low},
where we have observed if $y<x<0$, then $|x-y|\le |x+y|$ and 
\[ |x-y|\le |x+y+2u|,   \qquad \forall u<0 . \]

For $I^c_2$ if following the same line one would have to require for all $y<x<0$ and $u<0$
 \[ |x-y|\le |x+y-2u| , \]   %\qquad \forall u<0,\] 
which is unfortunately not valid. Here we proceed by exploiting the expansion of $m_-(x,-k)$ as follows.
 %because of the phase $m_-(x,k)$
%\footnote{\begin{align*} &\Vert (x-y)I_2(x,y)\chi_{y<x<0}\Vert_2=\Vert \chi_{y<x<0} \int  \psi_j(k) m_-(x,-k) m_-(y,k) e^{i(x-y)k}dk\Vert_2\\ 
 %=&\Vert (x-y)\chi_{y<x<0} \int%_{|k|\sim\lam^{-1}}\psi_j(k)\big(1+\int_{-\infty}^0 B_-(x,u)e^{2iku}du \big) m_-(y,k)e^{i(x-y)k}dk\Vert_2\\
%\lesssim& \Vert (x-y)\chi_{y<x<0} \int%_{|k|\sim\lam^{-1}}\psi_j(k) m_-(y,k)e^{i(x-y)k}dk\Vert_2\\ +&\Vert (x-y)\chi_{y<x<0}\int_{-\infty}^0 B_-(x,u)du\int\psi_j(k) m_-(y,k)e^{i(x-y+2u)k}dk\Vert_2\\ :=&I_{21}+I_{22}.\end{align*} $I_{21}$ ok. \crr $I_{22}$ problematic} 
Iterating (\ref{e:m-hVm-}) we write with $t_0=x$\begin{align*}
&m_-(x,-k)
%=&1+\int_{-\infty}^x h(x-t,-k)V(t)dt+\int_{-\infty}^x h(x-t_1,-k)V(t_1)dt_1\int^{t_1}_{-\iy}h(t_1-t_2,-k)V(t_2)dt_2+\cdots\\
=1+\sum_{n=1}^\iy \int_{-\infty}^{t_0} h(t_0-t_1,-k)V(t_1)dt_1\\
\times&\int^{t_1}_{-\iy} h(t_1-t_2,-k)V(t_2)dt_2\cdots
\int_{-\infty}^{t_{n-1}} h(t_{n-1}-t_n,-k)V(t_n)dt_n\\
%=&1+\sum_{n=1}^\iy %\int\psi_j(k)r_-(k) m_-(y,k)e^{-i(x+y)k}dk\\
%\int_{-\infty}^{t_0=x} (\int_0^{t_0-t_1}e^{-2iku_1}du_1)V(t_1)dt_1\\
%\times&\int^{t_1}_{-\iy}(\int_0^{t_1-t_2}e^{-2iku_2}du_2)V(t_2)dt_2\cdots
%\int_{-\infty}^{t_{n-1}} (\int_0^{t_{n-1}-t_n}e^{-2iku_n}du_n)V(t_n)dt_n
:=&\sum_{n=0}^\iy M^-_{n}(x,k) .
\end{align*} 
Observe that 
\begin{align}
&h(x-t,-k)=\int_0^{x-t}e^{-2iku}du\label{e:h-green}\\
&\pa_k h(x-t,-k)%\left(\int_0^{x-t}e^{-2iku}du\right)=\int_0^{x-t}(-2iu)e^{-2iku}du\\ 
%=&\frac{1}{k}\int_0^{x-t}ud_u(e^{-2iku})\\
%=&\frac{1}{k}\left(ue^{-2iku}\vert_0^{x-t}-\int_0^{x-t}e^{-2iku} du\right)\\
=\frac{1}{k} \left((x-t)e^{-2ik(x-t)}-\int_0^{x-t}e^{-2iku} du\right) . \label{e:der-h-green} 
\end{align} 

We have by integration by parts \begin{align*}
%\Vert \chi_{\{y<x<0\}} 
-i(x-y)I^c_2(x,y)%\Vert_2
=&\sum_{n=0}^\iy%\Vert
%\chi_{y<x<0}
 \int\pa_k[\psi_j(k) m_-(y,k)M^-_{n}(x,k)] e^{i(x-y)k} dk\\ %\Vert_2 
%=&\sum_{n=0}^\iy \Vert (x-y)\chi_{y<x<0} \int \psi_j(k) m_-(y,k)e^{i(x-y)k}dk\\
%&\int_{-\infty}^{t_0} (\int_0^{t_0-t_1}e^{-2iku_1}du_1)V(t_1)dt_1\int^{t_1}_{-\iy}(\int_0^{t_1-t_2}e^{-2iku_2}du_2)V(t_2)dt_2\cdots\\
%\times&\int_{-\infty}^{t_{n-1}} (\int_0^{t_{n-1}-t_n}e^{-2iku_n}du_n)V(t_n)dt_n\Vert_2\\
%=&\sum_{n=0}^\iy \Vert (x-y)\chi_{y<x<0} \int \psi_j(k) m_-(y,k)e^{i(x-y)k}dk\\
 % &\int_{-\infty}^{t_0} h(t_0-t_1,-k)V(t_1)dt_1\int^{t_1}_{-\iy}h(t_1-t_2,-k)V(t_2)dt_2\cdots\\
%\times&\int_{-\infty}^{t_{n-1}} h(t_{n-1}-t_n,-k)V(t_n)dt_n\Vert_2
:=&\sum_{n=0}^\iy A^-_{n}(x,y) .  %negative phase 
\end{align*}

For $y<0$, $j\in\Z$ it is easy to see from Lemma \ref{l:der-m+} that 
\begin{align*}
\Vert A^-_0(x,y)\Vert_{L^2_x}=\Vert\psi_j(k)m_-(y,k)\Vert_{\dot{H}^1_k} \le c 2^{-j/4}\Vert \phi\Vert_{H^1([\frac14,1])} \,. 
\end{align*} 
For $n\ge 1$ using (\ref{e:h-green}) and exchanging order of integration give that
\begin{align*}
&A^-_{n}(x,y)%= \int\psi_j(k) m_-(y,k)d_k[e^{i(x-y)k}]\\
 %&\int_{-\infty}^{t_0} h(t_0-t_1,-k)V(t_1)dt_1\int^{t_1}_{-\iy}h(t_1-t_2,-k)V(t_2)dt_2\cdots\\
%\times&\int_{-\infty}^{t_{n-1}} h(t_{n-1}-t_n,-k)V(t_n)dt_n\\
=\int_{-\iy}^{t_0}V(t_1)dt_1\int_{0}^{t_0-t_1}du_1\cdots \int_{-\infty}^{t_{n-1}}V(t_n)dt_n \int_0^{t_{n-1}-t_n}du_n\\
\times&\int \pa_k[\psi_jm_-(y,k)] e^{i(x-y-2u_1-\cdots-2u_n)k}dk\\
+&\int\psi_j(k) m_-(y,k) \pa_kM_n^-(x,k)e^{i(x-y)k} dk:=\Pi_1(x,y)+\Pi_2(x,y) , \end{align*}
where \begin{align*} &\pa_kM_n^-(x,k)\\
=&\int_{t_0>t_1>t_2>\cdots>t_n} \pa_k\bigg(h(t_0-t_1,-k)\cdots h(t_{n-1}-t_n,-k)\bigg)\\
\times& V(t_1)\cdots V(t_n)dt_1\cdots dt_n
:=J^-_{n,1}+J^-_{n,2}+\cdots+J^-_{n,n} \,,
\end{align*}
 $J^-_{n,i}$ denoting the integral involving $\pa_kh(t_{i-1}-t_i,-k)$, $i=1,\dots,n$.

We estimate by Minkowski inequality 
\begin{align*} &\Vert \chi_{\{y<x<0\}}\Pi_1(x,y)\Vert_{L^2_x}\le \int_{-\iy}^{0}|V(t_1)|dt_1\int_{0}^{-t_1}du_1\cdots \int_{-\infty}^{t_{n-1}}|V(t_n)|dt_n \int_0^{-t_n}du_n\\
\times&\Vert \int \pa_k[\psi_jm_-(y,k)] e^{i(x-y-2u_1-\cdots-2u_n)k}dk\Vert_{L_x^2}\\
%\le&c(\Vert V\Vert_{L^1_1})2^{-j/4}\Vert \phi\Vert_{H^1[\frac14,1]} \int_{-\iy}^{0}(-t_1)|V(t_1)|dt_1\cdots \int_{-\infty}^{t_{n-1}}|V(t_n)| (t_{n-1}-t_n)dt_n\\
\le&c2^{-j/4}\Vert \phi\Vert_{H^1([\frac14,1])}\int_{-\iy}^{0}(-t_1)|V(t_1)|dt_1\cdots \int_{-\infty}^{t_{n-1}}(-t_n)|V(t_n)|dt_n\\
\le&c2^{-j/4}\Vert \phi\Vert_{H^1([\frac14,1])}\frac{(\Vert tV\Vert_1)^n}{n!} \,. %\quad {check} 
\end{align*}
For $\Pi_2$ we estimate the first term by using (\ref{e:der-h-green}), Minkowski inequality and Plancherel theorem %for Fourier 
to obtain
\begin{align*}
 &\Vert \chi_{\{y<x<0\}}\int\psi_j(k)m_-(y,k) J^-_{n,1}(x,k)e^{i(x-y)k}dk\Vert_{L^2_x}\\
\le&\int_{-\infty}^{0} 
(-t_1)|V(t_1)|dt_1\int^{t_1}_{-\infty}|V(t_2)|dt_2\int_0^{t_1-t_2}du_2\cdots
 \int_{-\iy}^{t_{n-1}}|V(t_n)|dt_n\int_0^{t_{n-1}-t_n}du_n\\
 &\Vert\int\frac{\psi_j(k)}{k}m_-(y,k) e^{-i(x+y-2t_1+2u_2\cdots+2u_n)k} dk\Vert_{L^2_x}\\
+&\int_{-\infty}^{0} |V(t_1)|dt_1\int_0^{-t_1}du_1\int^{t_1}_{-\infty}|V(t_2)|dt_2\int_0^{t_1-t_2}du_2\cdots
 \int_{-\iy}^{t_{n-1}}|V(t_n)|dt_n\int_0^{t_{n-1}-t_n}du_n\\
&\Vert\int\frac{\psi_j(k)}{k}m_-(y,k) e^{i(x-y-2u_1-2u_2\cdots-2u_n)k}dk \Vert_{L^2_x} \\
\le& c
 2^{-j/4} \Vert \phi\Vert_{L^2([\frac14,1])}
 \frac{(\Vert tV\Vert_{1})^n}{n!} \,.
\end{align*}
The same estimate holds for other terms involving $J^-_{n,i}$, $i=2,\dots,n$.
And so, \begin{align*} \Vert \chi_{\{y<x<0\}}A_n^-(x,y)\Vert_2%\le \Vert \chi_{\{y<x<0\}}\Pi_1(x,y)\Vert_2+\Vert 
\le c 2^{-j/4} 
 (1+n)\frac{(\Vert tV\Vert_{1})^n}{n!}\Vert \phi\Vert_{H^1([\frac14,1])} \,.
\end{align*}
It follows that for all $j\in\Z$ \begin{align*} \Vert \chi_{\{y<x<0\}}(x-y)I^c_2(x,y)\Vert_2%\le \sum_{n=0}^\iy 
\le c
2^{-j/4}e^{\Vert tV\Vert_{1}}\Vert\phi\Vert_{H^1([\frac14,1])} \,,
\end{align*}
which proves (\ref{e:w-L2}c). 
\end{proof}

%\subsection{Unweighted $L^2$ estimate. The case $s=0$} 
\begin{lemma}\label{l:j-ker-L2} Let $V\in L^1_1$. Then for all $y$, $j\in\Z$ and $\phi\in L^2([\frac14,1])$%$s=0$
\begin{align*} 
&\Vert \phi_j(H_{ac})(x,y)\Vert_{L^2_x} \le 
 c 2^{j/4}\Vert \phi\Vert_{L^2([\frac14,1])} \,.\label{e:kerj-L2} 
\end{align*}
\end{lemma}
The proof is straightforward and follows the same line as in the weighted case $s=1$
but much simpler, where we only need Lemma \ref{l:B-L1} (a) and the following asymptotics 
 in Lemma \ref{l:der-m+} and Lemma \ref{l:trm-asym-low}: 
If $V\in L^1_1$,  then
\[  \begin{cases} 
|m_\pm(y,k)|\le&c(1+\max(0,\mp y))\\
t(k)=&O(1)  \\
r_+(k)=&O(1) .
\end{cases}\]  
We omit the details.  

\vs{.23in}
\nd{\bf Aknowledgment.} The  author would like to thank  the referee for careful reading and 
kind suggestions on the original manuscript. % that have helped improve the readability

\end{document}